\documentclass{jT}

\usepackage{amssymb,amsmath,latexsym,amsthm}
\usepackage[english]{babel}

\newtheorem{theorem}{Theorem}[section]
\newtheorem{lemma}{Lemma}[section]
\newtheorem{proposition}{Proposition}[section]
\newtheorem{corollary}{Corollary}[section]

\theoremstyle{definition}
\newtheorem*{definition}{Definition}

\numberwithin{equation}{section}

\def\as#1{\renewcommand\arraystretch{#1}}
\def\gg{\mathcal{G}r}
\def\ggm{\mathcal{G}r(\mu)}
\def\ggmp{\mathcal{G}r(\mu')}
\def\hh{{\mathcal H}}
\def\hm{H_\mu}
\def\hmp{H_{\mu'}}
\def\imp{\,\Longrightarrow\,}
\def\iso{\ \lower.3ex\hbox{\as{.08}$\begin{array}{c}\lra\\\mbox{\tiny $\sim\,$}\end{array}$}\ }
\def\kb{\overline{K}_v}
\def\kpm{\op{KP}(\mu)}
\def\la{\lambda}
\def\ll{\mathcal{L}}
\def\lra{\longrightarrow}
\def\m{{\mathfrak m}}
\def\md#1{\; \mbox{\rm(mod }{#1})}
\def\mmu{\mid_\mu}
\def\mn{\op{Min}}
\def\mx{\op{Max}}
\def\oo{\mathcal{O}}
\def\op{\operatorname}
\def\ord{\op{ord}}

\def\ppa{\mathcal{P}_{\alpha}}
\def\pset{\mathcal{P}}
\def\rep{\operatorname{Rep}}
\def\res{\operatorname{Res}}
\def\rr{\mathcal{R}}
\def\rrm{\mathcal{R}_\mu}
\def\sii{\,\Longleftrightarrow\,}
\def\smu{\sim_\mu}
\def\t{\theta}

\def\tst{\T^{\operatorname{str}}}
\def\ty{\mathbf{t}}
\newcommand{\F}{\mathbb F}
\def\P{\mathbb P}
\newcommand{\Q}{\mathbb Q}
\newcommand{\R}{\mathbb R}
\newcommand{\T}{\mathbb T}
\newcommand{\V}{\mathbb V}
\newcommand{\Z}{\mathbb Z}

\begin{document}

\title[On the equivalence of types]{On the equivalence of types}

\author{\sc Enric Nart}
\address{Enric Nart\\
Departament de Matem\`atiques, Universitat Aut\`onoma de Barcelona, E-08193 Bellaterra,  Cerdanyola del Vall\`es, Barcelona\\
Catalonia, Spain}
\email{nart@mat.uab.cat}



\thanks{Partially supported by MTM2013-40680-P from the Spanish MEC}

\maketitle

\begin{resume}
Un type sur un corps de valuation discr\`ete $(K,v)$ est un objet computationnel qui param\`etrise une famille de polyn\^omes unitaires irr\'eductibles sur $K_v[x]$, o\`u $K_v$ est le compl\'et\'e de $K$. Deux types sont \'equivalents s'ils determinent la m\^eme famille de polyn\^omes irr\'eductibles sur $K_v[x]$. Dans ce travail, nous donnons diff\'erentes caract\'erisations de la notion d'\'equivalence de types par rapport \`a certaines donn\'ees et des op\'erateurs qui leur sont associ\'es.   
\end{resume}

\begin{abstr}
Types over a discrete valued field $(K,v)$ are computational objects that parameterize certain families of monic irreducible polynomials in $K_v[x]$, where $K_v$ is the completion of $K$ at $v$. Two types are considered to be equivalent if they encode the same family of prime polynomials in $K_v[x]$. In this paper, we find diferent characterizations of the equivalence of types in terms of certain data and operators associated with them.  
\end{abstr}

\bigskip
\section{Introduction}

In the 1920's, \O.\! Ore developed a method to construct the prime ideals of a number field dividing a given prime number $p$, in terms of a defining polynomial $f\in\Z[x]$ satisfying a certain \emph{$p$-regularity} condition \cite{ore1, ore2}. The idea was to detect a $p$-adic factorization of $f$ from the factorization of certain residual polynomials over finite fields, attached to the sides of a Newton polygon of $f$. 
He raised then the question of the existence of a procedure to compute the prime ideals in the general case, based on the consideration of similar Newton polygons and residual polynomials ``of higher order". 

S. MacLane solved this problem in 1936 in a more general context \cite{mcla,mclb}. For any discrete valuation $v$ on an arbitrary field $K$, he described all valuations extending $v$ to the rational function field $K(x)$. Starting from the Gauss valuation $\mu_0$, MacLane constructed \emph{inductive} valuations $\mu$ on $K(x)$ extending $v$, by the concatenation of augmentation steps
\begin{equation}\label{depth0}
\mu_0\ \stackrel{(\phi_1,\nu_1)}\lra\  \mu_1\ \stackrel{(\phi_2,\nu_2)}\lra\ \cdots\ \lra\ \mu_{r-1} \ \stackrel{(\phi_{r},\nu_{r})}\lra\ \mu_{r}=\mu,
\end{equation} 
based on the choice of certain \emph{key polynomials} $\phi_i\in K[x]$ and positive rational numbers $\nu_i$.
Then, given an irreducible polynomial $f\in K[x]$, he characterized all extensions of $v$ to the field $L:=K[x]/(f)$ as limits of sequences of inductive valuations on $K(x)$ whose value at $f$ grows to infinity.  In the case $K=\Q$, Ore's $p$-regularity condition is satisfied when all valuations on the number field $L$ extending the $p$-adic valuation are sufficiently close to an inductive valuation on $K(x)$ which may be obtained from $\mu_0$ by a single augmentation step. 

In 1999, J. Montes carried out Ore's program in its original formulation \cite{HN,Mo}. He introduced \emph{types} as computational objects which are able to construct MacLane's valuations and the higher residual polynomial operators foreseen by Ore. These ideas made the whole theory constructive and well-suited to computational applications, and led to the design of several fast algorithms to perform arithmetic tasks in global fields \cite{algorithm,newapp,bases,GNP,Ndiff}.

In 2007, M. Vaqui\'e reviewed and generalized MacLane's work to non-discrete valuations. The introduction of the graded algebra $\ggm$ of a valuation $\mu$ led him to a more elegant presentation of the theory \cite{Vaq}.  

In the papers \cite{Rid} and \cite{gen}, which deal only with discrete valuations, the ideas of Montes were used to develop a constructive treatment of Vaqui\'e's approach, which included the computation of generators of the graded algebras and a thorough revision and simplification of the algorithmic applications. 

In this paper we fill a gap concerning the notion of \emph{equivalence of types}. Let $\oo\subset K$ be the valuation ring of $v$ and $\F$ its residue class field. A \emph{type over} $(K,v)$ is an object carrying certain data distributed into several levels:
$$
\ty=(\psi_0;(\phi_1,\nu_1,\psi_1);\dots;(\phi_r,\nu_r,\psi_r)).
$$
The pairs $\phi_i,\nu_i$ determine an inductive valuation $\mu_\ty:=\mu$ as in (\ref{depth0}), and  $\psi_i\in\F_i[y]$ are monic irreducible polynomials building a tower of finite extensions of $\F$:  
$$
\F=\F_0\lra \F_1\lra \cdots\lra \F_r,\quad\ \F_{i+1}:=\F_i[y]/(\psi_i),\ 0\le i<r.
$$
These data facilitate a recurrent procedure to construct \emph{residual polynomial operators}:
$$
R_i\colon K[x]\lra \F_i[y],\qquad 0\le i\le r, 
$$
having a key role in the theory. The last polynomial $\psi_r$ determines a maximal ideal $\ll_\ty$ of the piece of degree zero of the graded algebra $\gg(\mu_\ty)$. 

Two types are said to be equivalent when they yield the same pair $(\mu_\ty,\ll_\ty)$. This defines an equivalence relation $\equiv$ in the set  $\T$ of all types over $(K,v)$.

Any polynomial $g\in K[x]$ has an \emph{order of divisibility} by the type $\ty$, defined as $\ord_\ty(g):=\ord_{\psi_r}(R_r(g))$ in $\F_r[y]$. Let $\rep(\ty)$ be the set of all \emph{representatives} of $\ty$; that is, monic polynomials $\phi\in\oo[x]$ with minimal degree satisfying $\ord_\ty(\phi)=1$. We have $\rep(\ty)\subset\P$, where $\P$ is the set of monic irreducible polynomials with coefficients in $\oo_v$.

The main result of the paper states that two types $\ty$, $\ty^*$ are equivalent if and only if  $\rep(\ty)=\rep(\ty^*)$ (cf. Theorem \ref{finalchar}).    

The outline of the paper is as follows. In section \ref{secML} we recall some essential facts on MacLane valuations. In section \ref{secComparison} we analyze to what extent different chains  of augmentation steps as in (\ref{depth0}) may build the same valuation $\mu$. In section \ref{secTypes} we find a concrete procedure to decide whether two given types are equivalent, in terms of the data supported by them, and we describe then the relationship between their residual polynomial operators (Lemma \ref{optstep} and Proposition \ref{charequiv}). Finally, we prove Theorem \ref{finalchar}, which yields two more conceptual characterizations of the equivalence of types.  

Finally, let us add some remarks on the incidence of these results in the algorithmic applications of types and MacLane's valuations.

On the set $\P$ we may consider the following equivalence relation: two prime polynomials $F,G\in\P$ are \emph{Okutsu equivalent}, and we write $F\approx G$, if $v(\res(F,G))$ is greater than certain \emph{Okutsu bound} \cite[Sec. 4]{okutsu}, \cite{Ok}. Equivalence of types had been considered in \cite{gen} only for \emph{strongly optimal} types, which form a very special subset $\tst\subset \T$ (cf. section \ref{secTypes}). In \cite[Thm. 3.9]{gen} it is shown that the assignment $\ty\mapsto\rep(\ty)$ induces a canonical bijection
\begin{equation}\label{main}
 \tst/\equiv \ \longrightarrow \ \P/\approx,
\end{equation}
and the levels of $\ty\in\tst$ contain intrinsic data of the prime polynomials
in the Okutsu class of any representative of $\ty$.

Given a monic squarefree $f\in \oo[x]$, the \emph{Montes algorithm}, known also as the \emph{OM factorization algorithm}, computes a family of pairs $(\ty,\phi)$ parameterizing the prime factors of $f$ in $\oo_v[x]$. If a prime factor $F\in\oo_v[x]$ of $f$ is associated with a pair $(\ty,\phi)$, then $\ty$ is a strongly optimal type whose equivalence class is canonically attached to the Okutsu class of $F$ through the mapping of (\ref{main}), and $\phi\approx F$ is a concrete choice in $\oo[x]$ of a polynomial in the Okutsu class of $F$.  

However, the algorithm is based on the construction of certain non-optimal types, which must then be converted into optimal types in the same equivalence class. In the original presentation of the algorithm in \cite{algorithm,Mo}, the discussion of these optimization steps was based on some excruciating arguments, due to the absence of the concept of equivalence of types. Thus, the results of this paper contribute to a great simplification of the analysis of this optimization procedure. This is illustrated in section \ref{secExample}, where we present a concrete example of OM factorization.

\section{MacLane chains of inductive valuations}\label{secML}
Let $K$ be a field equipped with a discrete valuation $v\colon K^*\to \Z$, normalized so that $v(K^*)=\Z$. Let $\oo$ be the valuation ring of $K$, $\m$ the maximal ideal, $\pi\in\m$ a generator of $\m$ and $\F=\oo/\m$ the residue class field. 

Let $K_v$ be the completion of $K$ at $v$, with valuation ring $\oo_v\subset K_v$. Let $v\colon \kb^*\to \Q$ still denote the canonical extension of $v$ to a fixed algebraic closure of $K_v$. 

\subsection{Graded algebra of a valuation}\label{subsecGradedAlg}
Let $\V$ be the set of all discrete valuations $\mu\colon K(x)^*\to \Q$ such that $\mu_{\mid K}=v$ and $\mu(x)\ge0$. 

In the set $\V$ there is a natural partial ordering:
$$
\mu\le \mu' \quad\mbox{ if }\quad\mu(g) \le \mu'(g), \ \forall\,g\in K[x]. 
$$
Consider the Gauss valuation $\mu_0\in \V$ acting on polynomials as follows:
$$
\mu_0\left(\sum\nolimits_{0\le s}a_sx^s\right)=\mn_{0\le s}\left\{v(a_s)\right\}.
$$
Clearly, $\mu_0\le \mu$ for all $\mu\in\V$.

Let $\mu\in\V$ be a valuation. We denote by $\Gamma(\mu)=\mu\left(K(x)^*\right)\subset \Q$ the cyclic group of finite values of $\mu$. The \emph{ramification index} of $\mu$ is the positive integer $e(\mu)$ such that $e(\mu)\Gamma(\mu)=\Z$.   

For any $\alpha\in\Gamma(\mu)$ we consider the following $\oo$-submodules in $K[x]$:
$$
\ppa=\{g\in K[x]\mid \mu(g)\ge \alpha\}\supset
\ppa^+=\{g\in K[x]\mid \mu(g)> \alpha\}.
$$    

The \emph{graded algebra of $\mu$} is the integral domain:
$$
\ggm:=\bigoplus\nolimits_{\alpha\in\Gamma(\mu)}\ppa/\ppa^+.
$$

Let $\Delta(\mu)=\pset_0/\pset_0^+$ be the subring determined by the piece of degree zero of this algebra. Clearly, $\oo\subset\pset_0$ and $\m=\pset_0^+\cap \oo$; thus, there is a canonical homomorphism $\F\to\Delta(\mu)$, equipping  $\Delta(\mu)$ (and $\ggm$) with a canonical structure of $\F$-algebra. 


There is a natural map $\hm\colon K[x]\lra \ggm$, given by $\hm(0)=0$, and
$$\hm(g)= g+\pset_{\mu(g)}^+\in\pset_{\mu(g)}/\pset_{\mu(g)}^+,$$
for $g\ne0$. Note that $\hm(g)\ne0$ if $g\ne0$. For all $g,h\in K[x]$ we have:
$$
\begin{array}{l}
 \hm(gh)=\hm(g)\hm(h), \\
 \hm(g+h)=\hm(g)+\hm(h), \mbox{ if }\mu(g)=\mu(h)=\mu(g+h).
\end{array}
$$

If  $\mu\le \mu'$ for some $\mu'\in\V$, we have a canonical homomorphism of graded algebras 
$$\ggm\to\gg(\mu'),\qquad g+\ppa^+(\mu)\mapsto g+\ppa^+(\mu').$$ The image of $\hm(g)$ is $\hmp(g)$ if $\mu(g)=\mu'(g)$, and zero otherwise. 

\begin{definition}\label{mu}\mbox{\null}
Consider $g,\phi\in K[x]$.

We say that $g,\phi$ are \emph{$\mu$-equivalent}, and we write $g\smu \phi$, if $\hm(g)=\hm(\phi)$. 

We write $\phi\mmu g$, if $\hm(g)$ is divisible by $\hm(\phi)$ in $\ggm$. 

We say that $\phi$ is $\mu$-irreducible if $\hm(\phi)\ggm$ is a non-zero prime ideal.

We say that $\phi$ is $\mu$-minimal if $\phi\nmid_\mu h$ for all non-zero $h\in K[x]$ with $\deg h<\deg \phi$.
\end{definition}

\subsection{Augmentation of valuations}
A \emph{key polynomial} for the valuation $\mu$ is a monic polynomial in $K[x]$ which is $\mu$-minimal and $\mu$-irreducible. Let us denote by $\kpm$ the set of key polynomials for $\mu$.

Every key polynomial has coefficients in $\oo$ and is irreducible in $\oo_v[x]$ \cite[Lem. 1.8, Cor. 1.10]{Rid}. 

\begin{lemma}\cite[Lem. 1.4]{Rid}\label{mid=sim}
Consider $\phi\in\kpm$ and $g\in K[x]$ a monic polynomial such that $\phi\mmu g$ and $\deg g=\deg\phi$. Then, $\phi\smu g$ and $g$ is a key polynomial for $\mu$ too.
\end{lemma}

\begin{definition}\label{muprima}
Take $\phi\in \kpm$. For $g\in K[x]$ let $g=\sum_{0\le s}a_s\phi^s$ be its canonical $\phi$-expansion in $K[x]$, uniquely determined by the condition $\deg a_s<\deg\phi$
for all $s\ge 0$.  

Take $\nu\in \Q_{>0}$. The augmented valuation $\mu'=[\mu;\phi,\nu]$ with respect to the pair $\phi,\nu$ is the valuation $\mu'$ determined by the following action on $K[x]$:
 $$\mu'(g):=\mn_{0\le s}\{\mu(a_s\phi^s)+s\nu\}=\mn_{0\le s}\{\mu'(a_s\phi^s)\}.$$ 
\end{definition}

\begin{proposition}\cite[Prop. 1.7]{Rid}\label{extension}
\begin{enumerate}
\item 
The natural extension of $\mu'$ to $K(x)$ is a valuation and $\mu\le\mu'$.
\item
For a non-zero $g\in K[x]$, $\mu(g)=\mu'(g)$ if and only if $\phi\nmid_{\mu}g$. 
\item The polynomial $\phi$ is a key polynomial for $\mu'$ too.
\end{enumerate}
\end{proposition}

\begin{lemma}\cite[Lem. 3.5]{Rid}\label{unicity}
Let $\mu''=[\mu;\phi^*,\nu^*]$ be another augmentation of $\mu$. We have $\mu'=\mu''$ if and only if $\deg\phi^*=\deg\phi$, $\mu'(\phi^*-\phi)\ge \mu'(\phi)$,  and $\nu^*=\nu$. In this case, $\phi^*\smu\phi$. 
\end{lemma}

Denote $\Delta=\Delta(\mu)$, and let $I(\Delta)$ be the set of ideals in $\Delta$. Consider the following \emph{residual ideal operator}:
$$
\rr=\rrm\colon K[x]\lra I(\Delta),\qquad g\mapsto \Delta\cap \hm(g)\ggm.
$$

Let $\phi$ be a key polynomial for $\mu$. Choose a root $\t \in\kb$ of $\phi$ and denote by $K_\phi=K_v(\t)$ the finite extension of $K_v$ generated by $\t$. Also, let $\oo_\phi\subset K_\phi$ be the valuation ring of $K_\phi$, $\m_\phi$ the maximal ideal and $\F_\phi=\oo_\phi/\m_\phi$ the residue class field. 

\begin{proposition}\cite[Prop. 1.12]{Rid}\label{sameideal}
If $\phi$ is a key polynomial for $\mu$, then 
\begin{enumerate}
\item $\rr(\phi)$ is the kernel of the onto homomorphism $\Delta\twoheadrightarrow \F_\phi$ determined by $g+\pset^+_0\ \mapsto\ g(\t)+\m_\phi$. 
Hence, $\rr(\phi)$ is a maximal ideal of $\Delta$.
\item $\rr(\phi)=\op{Ker}(\Delta\to \Delta(\mu'))$ for any augmented valuation  $\mu'=[\mu;\phi,\nu]$. Thus, the image of $\Delta\to\Delta(\mu')$ is a field canonically isomorphic to $\F_\phi$. 
\end{enumerate}
\end{proposition}

The map $\rr\colon\kpm\to\mx(\Delta)$ is onto and its fibers coincide with the $\mu$-equivalence classes of key polynomials  \cite[Thm. 5.7]{Rid}:
\begin{equation}\label{repr}
\rr(\phi)=\rr(\phi^*) \sii \phi\smu\phi^*\sii \phi\mmu\phi^*.
\end{equation} 

\subsection{MacLane chains}\label{subsecML}
Let $\mu\in \V$ be an \emph{inductive valuation}; that is, $\mu$ may be obtained from the Gauss valuation $\mu_0$ by a finite number of augmentation steps:
\begin{equation}\label{depth}
\mu_0\ \stackrel{\phi_1,\nu_1}\lra\  \mu_1\ \stackrel{\phi_2,\nu_2}\lra\ \cdots
\ \stackrel{\phi_{r-1},\nu_{r-1}}\lra\ \mu_{r-1} 
\ \stackrel{\phi_{r},\nu_{r}}\lra\ \mu_{r}=\mu
\end{equation}
satisfying $\phi_{i+1}\nmid _{\mu_i}\phi_i$ for all $1\le i<r$. Such a chain of augmentations is called a \emph{MacLane chain} of $\mu$.
In a MacLane chain, the value group $\Gamma(\mu_{i})$ is the subgroup of \,$\Q$ generated by $\Gamma(\mu_{i-1})$ and $\nu_{i}$, for any $1\le i\le r$. In particular,
$$
\Z=\Gamma(\mu_0)\subset \Gamma(\mu_1)\subset\cdots\subset\Gamma(\mu_{r-1})\subset \Gamma(\mu_r)=\Gamma(\mu).
$$

A MacLane chain of $\mu$ supports several data and operators containing relevant information about $\mu$. Among them, the following deserve special mention:\medskip


\noindent{\bf (1) A sequence of finite field extensions of the residue class field:}
$$
\begin{array}{rcccccl}
\Delta_0&\lra&\Delta_1&\lra&\cdots&\lra&\Delta_r=\Delta(\mu)\\
\cup&&\cup&&\cdots&&\cup\\
\F=\F_0&\lra&\F_1&\lra&\cdots&\lra&\F_r
\end{array}
$$
where $\Delta_i=\Delta(\mu_i)$, the maps $\Delta_i\to\Delta_{i+1}$ are the canonical homomorphisms induced from the inequality $\mu_i\le\mu_{i+1}$, and $\F_i:=\op{Im}(\Delta_{i-1}\to\Delta_i)$.\medskip

\noindent {\bf (2) Numerical data.} Set $\phi_0=x$, $\nu_0=0$,  $\mu_{-1}=\mu_0$ and $\F_{-1}=\F_0$. 

For all $0\le i\le r$, we define integers:
$$
\begin{array}{lll}
e_i:=e(\mu_i)/e(\mu_{i-1}),&\quad f_{i-1}:=[\F_{i}\colon \F_{i-1}], &\quad h_i:=e(\mu_i)\nu_i,\\
m_i:=\deg\phi_i,&\quad V_i:=e(\mu_{i-1})\mu_{i-1}(\phi_i). &
\end{array}
$$
which satisfy the following relations for $1\le i\le r$:
\begin{equation}\label{recurrence}
\begin{array}{l}
\gcd(e_i,h_i)=1,\\
e(\phi_i)=e(\mu_{i-1})=e_0\cdots e_{i-1},\\
f(\phi_i)=\left[\F_i\colon \F_0\right]=f_0\cdots f_{i-1},\\
m_i=e_{i-1}f_{i-1}m_{i-1}=(e_0\cdots e_{i-1})(f_0\cdots f_{i-1}),\\
V_i=e_{i-1}f_{i-1}(e_{i-1}V_{i-1}+h_{i-1}),\\
\end{array}
\end{equation}
where $e(\phi_i)$, $f(\phi_i)$ denote the ramification index and residual degree of the finite extension $K_{\phi_i}/K_v$, respectively.\medskip

\noindent {\bf (3) Generators of the graded algebras}:
$$
p_i\in\gg(\mu_i)^*,\quad x_i\in\gg(\mu_i),\quad y_i\in\Delta_i,\quad 0\le i\le r,
$$ 
such that $\Delta_i=\F_i[y_i]$ and $\gg(\mu_i)=\Delta_i[p_i,p_i^{-1}][x_i]$. The elements $p_i,y_i$ are algebraically independent over $\F_i$ and $x_i$ satisfies the algebraic relation $x_i^{e_i}=y_ip_i^{h_i}$.
In particular, we have a family of $\F_i$-isomorphisms:
\begin{equation}\label{ji}
 j_i\colon \F_i[y]\lra \Delta_i,\quad y\mapsto y_i,\quad 0\le i\le r.
\end{equation}

Starting with $p_0=\op{H}_{\mu_0}(\pi)$, the generators are defined by the following recurrent relations:
$$
x_i=\op{H}_{\mu_i}(\phi_i)p_i^{-V_i},\quad 
y_i=x_i^{e_i}p_i^{-h_i},\quad
p_{i+1}=x_i^{\ell_i}p_i^{\ell'_i},
$$
where $\ell_i$, $\ell'_i\in\Z$ are uniquely determined by $\ell_ih_i+\ell'_ie_i=1$ and $0\le \ell_i<e_i$. In the relation concerning $p_{i+1}$ we identify the elements $x_i$, $p_i$ with their images under the canonical homomorphism $\gg(\mu_i)\to\gg(\mu_{i+1})$. \medskip

\noindent {\bf (4) Newton polygon operators}:
$$
N_i:=N_{\mu_{i-1},\phi_i}\colon K[x]\lra 2^{\R^2},\quad 1\le i\le r.
$$
For any nonzero $g\in K[x]$ consider its canonical $\phi$-expansion $g=\sum_{0\le s}a_s\phi^s$, where $a_s\in K[x]$ have $\deg a_s<\deg\phi$. Then, $N_i(g)$ is the lower convex hull of the set of points $\left\{(s,\mu_{i-1}(a_s\phi_i^s))\mid s\ge0\right\}$ in the Euclidean plane.\medskip

\noindent {\bf (5) Residual polynomial operators}:
$$
R_i:=R_{\mu_{i-1},\phi_i,\nu_i}\colon K[x]\lra \F_i[y], \quad 0\le i\le r,
$$
uniquely determined by the condition:
\begin{equation}\label{mainR}
\op{H}_{\mu_i}(g)=x_i^{s_i(g)}p_i^{u_i(g)}R_i(g)(y_i),
\end{equation}
for all nonzero $g\in K[x]$. For $i=0$ we define $s_0(g)=0$, $u_0(g)=\mu_0(g)$. For $i>0$, the point $(s_i(g),u_i(g)/e(\mu_{i-1}))$ is the left end point of the \emph{$\nu_i$-component} $S_{\nu_i}(g)$ of the Newton polygon $N_i(g)$, which is defined as the intersection of $N_i(g)$ with the line of slope $-\nu_i$ first touching the polygon from below (see Figure \ref{figComponent}). 

Let $s_i(g)\le s'_i(g)$ be the abscissas of the left end points of $S_{\nu_i}(g)$. The
 polynomial $R_i(g)$ has degree $(s'_i(g)-s_i(g))/e_i$, nonzero constant term, and it determines a generator of the residual ideal $\rr_{\mu_i}(g)$ as follows:
$\rr_{\mu_i}(g)=y_i^{\lceil s_i(g)/e_i\rceil}R_i(g)(y_i)\Delta_i$. \medskip

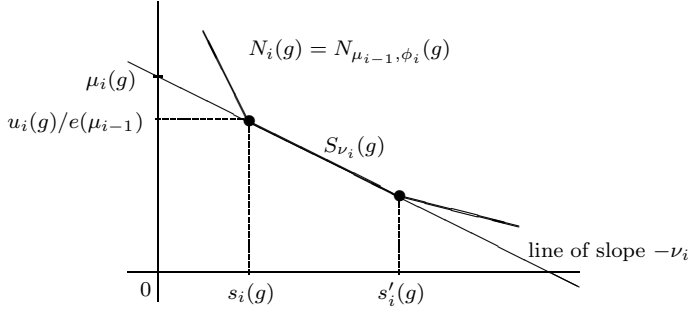
\begin{figure}
\caption{$\nu_i$-component $S_{\nu_i}(g)$ of $N_i(g)$ for $g\in K[x]$.}\label{figComponent}
\begin{center}
\setlength{\unitlength}{4mm}
\begin{picture}(14,10)
\put(2.8,5.8){$\bullet$}\put(7.8,3.3){$\bullet$}
\put(-1,1){\line(1,0){15}}\put(0,0){\line(0,1){10}}
\put(-1,8){\line(2,-1){15}}
\put(3,6){\line(-1,2){1.5}}\put(3,6.04){\line(-1,2){1.5}}
\put(3,6){\line(2,-1){5}}\put(3,6.04){\line(2,-1){5}}
\put(8,3.5){\line(4,-1){4}}\put(8,3.54){\line(4,-1){4}}
\multiput(3,.9)(0,.25){21}{\vrule height2pt}
\multiput(8,.9)(0,.25){11}{\vrule height2pt}
\multiput(-.1,6.05)(.25,0){13}{\hbox to 2pt{\hrulefill }}
\put(7.3,.1){\begin{footnotesize}$s'_i(g)$\end{footnotesize}}
\put(2.3,.1){\begin{footnotesize}$s_i(g)$\end{footnotesize}}
\put(-5,5.8){\begin{footnotesize}$u_i(g)/e(\mu_{i-1})$\end{footnotesize}}
\put(12.3,1.5){\begin{footnotesize}line of slope $-\nu_i$\end{footnotesize}}
\put(-.6,.2){\begin{footnotesize}$0$\end{footnotesize}}
\put(3,8.2){\begin{footnotesize}$N_i(g)=N_{\mu_{i-1},\phi_i}(g)$\end{footnotesize}}
\put(5.5,5){\begin{footnotesize}$S_{\nu_i}(g)$\end{footnotesize}}
\put(-.15,7.5){\line(1,0){.3}}
\put(-2.4,7.2){\begin{footnotesize}$\mu_i(g)$\end{footnotesize}}
\end{picture}
\end{center}
\end{figure}

\noindent {\bf (6) A family of maximal ideals} $\ll_i\in\mx(\Delta_i)$, for $0\le i<r$. The ideals $\ll_i$ are determined by Proposition \ref{sameideal} as:
$$
\ll_i:=\op{Ker}(\Delta_i\lra \Delta_{i+1})=\rr_{\mu_i}(\phi_{i+1}),\quad 0\le i<r.
$$ 
Through the isomorphisms $j_i$ of (\ref{ji}), these ideals yield monic irreducible polynomials  $\psi_i\in\F_i[y]$ uniquely determined by the condition $j_i(\psi_i\F_i[y])=\psi_i(y_i)\Delta_i=\ll_i$, or alternatively, by the condition $\psi_i=R_i(\phi_{i+1})$.
We have a commutative diagram with vertical isomorphisms:
$$\as{1.2}\begin{array}{ccl}\F_i[y]&\twoheadrightarrow&\F_i[y]/(\psi_i)\\j_i\downarrow\hphantom{m}&&\quad\downarrow\\\Delta_i&\twoheadrightarrow&\Delta_i/\ll_i\iso\F_{i+1}\subset \Delta_{i+1}\end{array}$$
Hence, $\deg\psi_i=[\F_{i+1}\colon\F_i]=f_i$, for $0\le i<r$.

\section{Data comparison between MacLane chains}\label{secComparison}

Consider a MacLane chain of an inductive valuation $\mu$ as in (\ref{depth}), supporting the data and operators described above. In this section, we analyze the variation of these data and operators when a different MacLane chain of the same valuation is chosen.

Note that $\F_r$ is the algebraic closure of $ \F$ in $\Delta:=\Delta(\mu)$, through the canonical map $\F\to\Delta$. Thus, this field does not depend on the choice of the MacLane chain. We may denote it by $\F_\mu:=\F_r$. It must not be confused with the residue class field $\kappa(\mu)$ of the valuation $\mu$. Actually, $\kappa(\mu)$ is isomorphic to the field of fractions of $\Delta$ \cite[Prop. 3.9]{Rid}, so that $\F_\mu$ is isomorphic to the algebraic closure of $\F$ in $\kappa(\mu)$ too.

\begin{definition}\label{proper}
A key polynomial $\phi \in\kpm$ is said to be \emph{proper} if $\mu$ admits a MacLane chain such that $\phi\nmid_\mu\phi_r$, where $r$ is the length of the chain and $\phi_r$ is the key polynomial of the last augmentation step.  
\end{definition}

For any MacLane chain of length $r$ of $\mu$, we have \cite[Sec. 5.3]{Rid}:
\begin{equation}\label{em}
\begin{array}{l}
m_r=\mn\left\{\deg\phi\mid\phi \in\kpm\right\}, \\
e_rm_r=\mn\left\{\deg\phi\mid\phi \in\kpm,\ \phi\mbox{ proper}\right\},\\
\phi\in\kpm \mbox{ proper }\sii \deg\phi\ge e_rm_r. 
\end{array}
\end{equation}

Thus, the positive integers $m_\mu:=m_r$, $e_\mu:=e_r$ do not depend on the choice of the MacLane chain either.

\subsection{Independence of the lower levels}
Our first aim is to prove the following result.

\begin{theorem}\label{laststep}
Let $\phi$ be a proper key polynomial for the inductive valuation $\mu$ and consider a MacLane chain of $\mu$ as in (\ref{depth}) with $\phi\nmid_{\mu}\phi_r$. For any $\nu\in\Q_{>0}$  consider the MacLane chain of the augmented valuation $\mu'=[\mu;\phi,\nu]$ obtained by adding one augmentation step:
\begin{equation}\label{extchain}
\mu_0\ \stackrel{\phi_1,\nu_1}\lra\  \mu_1\ \stackrel{\phi_2,\nu_2}\lra\ \cdots
\ \lra\ \mu_{r-1} 
\ \stackrel{\phi_{r},\nu_{r}}\lra\ \mu_{r}=\mu\ \stackrel{\phi,\nu}\lra\ \mu_{r+1}=\mu'
\end{equation}
Then, the elements
$$
p_{r+1}\in\ggmp^*,\quad x_{r+1}\in\ggmp,\quad y_{r+1}\in\Delta(\mu')
$$
and the operators $N_{r+1}$, $R_{r+1}$ attached to this extended MacLane chain do not depend on the initial MacLane chain.
\end{theorem}

In other words, the generators of $\ggmp$ and the operators $N_{r+1}$, $R_{r+1}$ depend on $\mu,\phi,\nu$, but not on the choice of a MacLane chain of $\mu$.
In particular, we obtain a residual polynomial operator
$$
R_{\mu,\phi,\nu}\colon K[x]\lra \F_{\mu'},
$$
defined as $R_{\mu,\phi,\nu}:=R_{r+1}$, which depends only on $\mu$, $\phi$ and $\nu$.

The proof of Theorem \ref{laststep} requires some previous work.

\begin{lemma}\label{augmentation}
Consider a MacLane chain of augmented valuations
$$
\mu^*\stackrel{\phi^*,\nu^*}\lra \mu\stackrel{\phi,\nu}\lra \mu' 
$$
with $\deg \phi=\deg\phi^*$. Then, $\phi\in\op{KP}(\mu^*)$ and $\mu'=[\mu^*;\phi,\nu^*+\nu]$. 

Further, consider the affine transformation
$$
\hh\colon \R^2\lra \R^2,\quad (x,y)\mapsto (x,y-\nu^*x).
$$ 
Then, $N_{\mu^*,\phi}=\hh\circ N_{\mu,\phi}$.
\end{lemma}

\begin{proof}
The first statement is just \cite[Lem. 3.4]{Rid}. For the comparison between  $N_{\mu^*,\phi}$ and $N_{\mu,\phi}$, consider the $\phi$-expansion $g=\sum_{0\le s}a_s\phi^s$ of a nonzero $g\in K[x]$. By the definition of the augmented valuation $\mu=[\mu^*;\phi^*,\nu^*]$,
$$
\deg a_s<\deg \phi=\deg\phi^*\imp \mu(a_s)=\mu^*(a_s). 
$$
On the other hand, $\phi^*=\phi+a$ for some $a\in K[x]$ with $\deg a<\deg \phi$. By hypothesis, $\phi\nmid_{\mu}\phi^*$, and this implies $\phi^*\nmid_{\mu}\phi$ by Lemma \ref{mid=sim}. Since $\phi$ and $ \phi^*$ are both $\mu$-minimal, \cite[Lem. 1.3]{Rid} shows that
$$
\mu(\phi)=\mu(a)=\mu(\phi^*)=\mu^*(\phi^*)+\nu^*. 
$$ 
Since $\mu^*(a)=\mu(a)>\mu^*(\phi^*)$, we deduce that $\mu^*(\phi)=\mu^*(\phi^*)=\mu(\phi)-\nu^*$. Thus, for each $s\ge0$ we have $\mu^*(a_s\phi^s)=\mu(a_s\phi^s)-s\nu^*$, or equivalently, $\hh(s,\mu(a_s\phi^s))=(s,\mu^*(a_s\phi^s))$.
\end{proof}

\begin{figure}
\caption{Comparison of Newton polygons of $g\in K[x]$}\label{figNcomparison}
\begin{center}
\setlength{\unitlength}{4mm}
\begin{picture}(14,13)
\put(-.25,10.8){$\bullet$}\put(1.8,6.75){$\bullet$}\put(1.8,8.75){$\bullet$}
\put(4.8,2.75){$\bullet$}\put(4.8,7.75){$\bullet$}
\put(6.8,.75){$\bullet$}\put(6.8,7.75){$\bullet$}
\put(-1,0){\line(1,0){11}}\put(0,-1){\line(0,1){13}}
\put(0,11){\line(1,-1){2}}\put(0,11){\line(1,-2){2}}
\put(0,11.02){\line(1,-1){2}}\put(0,11.02){\line(1,-2){2}}
\put(2,9.02){\line(3,-1){3}}\put(2,7.02){\line(3,-4){3}}
\put(2,9){\line(3,-1){3}}\put(2,7){\line(3,-4){3}}
\put(5,8){\line(1,0){2}}\put(5,3){\line(1,-1){2}}
\put(5,8.02){\line(1,0){2}}\put(5,3.02){\line(1,-1){2}}
\put(5,7.5){\vector(0,1){.45}}\put(5,3.5){\vector(0,-1){.45}}
\multiput(2,-.1)(0,.25){36}{\vrule height2pt}
\multiput(5,-.1)(0,.25){33}{\vrule height2pt}
\multiput(7,-.1)(0,.25){33}{\vrule height2pt}
\put(4.8,-.7){\begin{footnotesize}$s$\end{footnotesize}}
\put(5.2,5.2){\begin{footnotesize}$s\nu^*$\end{footnotesize}}
\put(8,7.8){\begin{footnotesize}$N_{\mu,\phi}(g)$\end{footnotesize}}
\put(8,.8){\begin{footnotesize}$N_{\mu^*,\phi}(g)$\end{footnotesize}}
\put(-.5,-.7){\begin{footnotesize}$0$\end{footnotesize}}
\end{picture}
\end{center}
\end{figure}
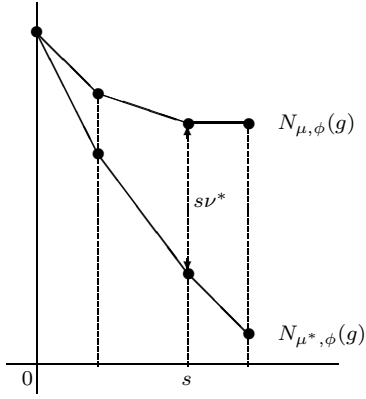

The affinity $\hh$ acts as a translation on every vertical line and it keeps the vertical axis pointwise invariant. Thus, a side $S$ of slope $\rho$ of $N_{\mu,\phi}(g)$ is mapped to a side of slope $\rho-\nu^*$ of  $N_{\mu^*,\phi}(g)$, whose end points have the same abscissas as those of $S$ (see Figure \ref{figNcomparison}). 

Let us now consider a very particular instance of Theorem \ref{laststep}. With the notation of that theorem, suppose that $r\ge 2$ and $\deg \phi_{r-1}=\deg\phi_r$, or equivalently, $e_{r-1}=f_{r-1}=1$. In this case, Lemma \ref{augmentation} shows that 
$\phi_r$ is a key polynomial for $\mu_{r-2}$ and $\mu=\mu_r=[\mu_{r-2};\phi_r,\nu_{r-1}+\nu_r]$ can be obtained as a simple augmentation of $\mu_{r-2}$.

Thus, we may consider two different MacLane chains of $\mu$:
$$
\begin{array}{l}
\mu_0\ \stackrel{\phi_1,\nu_1}\lra\  \mu_1\ \stackrel{\phi_2,\nu_2}\lra\ \cdots\ \lra \ \mu_{r-2}
\ \stackrel{\phi_{r-1},\nu_{r-1}}\lra\ \mu_{r-1} 
\ \stackrel{\phi_{r},\nu_{r}}\lra\ \mu_{r}=\mu\\
\mu^*_0\ \stackrel{\phi^*_1,\nu^*_1}\lra\  \mu^*_1\ \stackrel{\phi^*_2,\nu^*_2}\lra\ \cdots\ \lra \ \mu^*_{r-2}
\ \stackrel{\phi^*_{r-1},\nu^*_{r-1}}\lra\ \mu^*_{r-1}=\mu 
\end{array}
$$
where $\phi^*_{r-1}=\phi_r$, $\nu^*_{r-1}=\nu_{r-1}+\nu_r$, and
$$
\mu^*_i=\mu_i,\quad \phi^*_i=\phi_i,\quad \nu^*_i=\nu_i,\quad 0\le i\le r-2.
$$
We use the standard notation for all data and operators attached to the upper MacLane chain and we mark with a superscript $(\ )^*$ all data and operators attached to the lower one.

\begin{lemma}\label{firstcase}
With the above notation, let $\hh(x,y)=(x,y-\nu_{r-1}x)$.
\begin{enumerate}
\item $p^*_{r-1}=p_r$, \quad$x^*_{r-1}=x_rp_r^{h_{r-1}}$, \quad$y^*_{r-1}=y_r$.  
\item $N^*_{r-1}=\hh\circ N_r$, \quad$R^*_{r-1}=R_r$. 
\end{enumerate}
\end{lemma}

\begin{proof}
The generators of the graded algebra of $\mu$ were defined as follows:
$$
\as{1.4}
\begin{array}{ll}
p_r=x_{r-1}^{\ell_{r-1}}p_{r-1}^{\ell'_{r-1}},&p^*_{r-1}=p_{r-1}=x_{r-2}^{\ell_{r-2}}p_{r-2}^{\ell'_{r-2}},\\
x_r=\hm(\phi_r)p_r^{-V_r},&x^*_{r-1}=\hm(\phi_r)(p^*_{r-1})^{-V^*_{r-1}},\\
y_r=x_r^{e_r}p_r^{-h_r},&y^*_{r-1}=(x^*_{r-1})^{e^*_{r-1}}(p^*_{r-1})^{-h^*_{r-1}}.
\end{array}
$$
By hypothesis, $e_{r-1}=1$, so that $\ell_{r-1}=0$, $\ell'_{r-1}=1$; hence, $p_r=p_{r-1}=p^*_{r-1}$. On the other hand, the recurrences (\ref{recurrence}) show that
$$
\begin{array}{lclcl}
V_r&=&e_{r-1}f_{r-1}(e_{r-1}V_{r-1}+h_{r-1})&=&V_{r-1}+h_{r-1},\\
V^*_{r-1}&=&e^*_{r-2}f^*_{r-2}(e^*_{r-2}V^*_{r-2}+h^*_{r-2})&=&V_{r-1},
\end{array}
$$ 
because for levels $i<r-1$ the data of the two MacLane chains coincide.
Hence, 
$$x_r=\hm(\phi_r)p_r^{-V_r}=\hm(\phi_r)p_r^{-V_{r-1}-h_{r-1}}=x^*_{r-1}p_r^{-h_{r-1}}.
$$
As mentioned at the beginning of the section, $e^*_{r-1}=e_\mu=e_r$. Hence, from the equalities:
$$h^*_{r-1}/(e^*_1\cdots e^*_{r-1})=\nu^*_{r-1}=\nu_{r-1}+\nu_r=h_r/(e_1\cdots e_r)+h_{r-1}/(e_1\cdots e_{r-1}),
$$we deduce $h^*_{r-1}=h_r+e_rh_{r-1}$. Therefore,
$$
y^*_{r-1}=(x^*_{r-1})^{e_r}(p^*_{r-1})^{-h^*_{r-1}}=x_r^{e_r}p_r^{e_rh_{r-1}}p_r^{-h_r-e_rh_{r-1}}=x_r^{e_r}p_r^{-h_r}=y_r.
$$
This ends the proof of (1).

By Lemma \ref{augmentation}, we have
$$
N^*_{r-1}=N_{\mu^*_{r-2},\phi^*_{r-1}}=N_{\mu_{r-2},\phi_r}=\hh\circ N_{\mu_{r-1},\phi_r}=\hh\circ N_r.
$$

For any nonzero $g\in K[x]$, the affinity $\hh$ sends the $\nu_r$-component of $N_r(g)$ to the $\nu^*_{r-1}$-component of $N^*_{r-1}(g)$; hence,
$$
\hh(s_r(g),u_r(g)/e(\mu_{r-1}))=(s^*_{r-1}(g),u^*_{r-1}(g)/e(\mu^*_{r-2})).
$$
Having in mind that $e(\mu_{r-1})=e(\mu_{r-2})=e(\mu^*_{r-2})$, this shows that
\begin{equation}\label{claim}
s_r(g)=s^*_{r-1}(g),\qquad u_r(g)=u^*_{r-1}(g)+s_r(g)h_{r-1}. 
\end{equation}

Now, (\ref{mainR}) shows that
$$
x_r^{s_r(g)}p_r^{u_r(g)}R_r(g)(y_r)=H_{\mu}(g)=(x^*_{r-1})^{s^*_{r-1}(g)}(p^*_{r-1})^{u^*_{r-1}(g)}R^*_{r-1}(g)(y^*_{r-1}).
$$
From the identities in (\ref{claim}) and $x_r=x^*_{r-1}p_r^{-h_{r-1}}$ we deduce:
\begin{equation*}
\begin{split}
x_r^{s_r(g)}p_r^{u_r(g)}=&(x^*_{r-1})^{s_r(g)}p_r^{-s_r(g)h_{r-1}}p_r^{u^*_{r-1}(g)+s_r(g)h_{r-1}}\\=&
(x^*_{r-1})^{s^*_{r-1}(g)}(p^*_{r-1})^{u^*_{r-1}(g)}.
\end{split}
\end{equation*}

Therefore, $R_r(g)(y_r)=R^*_{r-1}(g)(y^*_{r-1})=R^*_{r-1}(g)(y_r)$ and this implies $R_r(g)=R^*_{r-1}(g)$ because $y_r$ is transcendental over $\F_r$ \cite[Thm. 4.3]{Rid}. This ends the proof of (2).
\end{proof}

These computations prove Theorem \ref{laststep} in this particular situation.

\begin{corollary}\label{firstproof}
With the above notation, let $\phi$ be a proper key polynomial for $\mu$ such that $\phi\nmid_\mu \phi_r$ and consider the augmented valuation $\mu'=[\mu;\phi,\nu]$. Then, the generators of $\ggmp$ and the operators $N_{\mu,\phi}$, $R_{\mu,\phi,\nu}$ attached to the following MacLane chains coincide. 
$$
\begin{array}{l}
\mu_0\ \stackrel{\phi_1,\nu_1}\lra\  \mu_1\ \stackrel{\phi_2,\nu_2}\lra\ \cdots\  
\ \stackrel{\phi_{r},\nu_{r}}\lra\ \mu_{r}=\mu\ \stackrel{\phi,\nu}\lra\ \mu_{r+1}=\mu'\\
\mu^*_0\ \stackrel{\phi^*_1,\nu^*_1}\lra\  \mu^*_1\ \stackrel{\phi^*_2,\nu^*_2}\lra\ \cdots\ 
\ \stackrel{\phi^*_{r-1},\nu^*_{r-1}}\lra\ \mu^*_{r-1}=\mu\ \stackrel{\phi,\nu}\lra\ \mu^*_r=\mu' 
\end{array}
$$
\end{corollary}

\begin{proof}
Let us compare the B\'ezout identities:
$$
\begin{array}{cl}
\ell_rh_r+\ell'_re_r=1,&\quad 0\le \ell_r<e_r,\\
\ell^*_{r-1}h^*_{r-1}+(\ell')^*_{r-1}e^*_{r-1}=1,&\quad 0\le \ell^*_{r-1}<e^*_{r-1}.
\end{array}
$$
From the identities $e^*_{r-1}=e_r$, $h^*_{r-1}=h_r+e_rh_{r-1}$, obtained during the proof of Lemma \ref{firstcase}, one deduces easily:
$$
\ell^*_{r-1}=\ell_r,\qquad (\ell')^*_{r-1}=\ell'_r-\ell_rh_{r-1}. 
$$

Let us denote $\phi=\phi_{r+1}=\phi^*_r$, $\nu=\nu_{r+1}=\nu^*_r$. Note that $e^*_r=e_{r+1}$
and  $h^*_r=h_{r+1}$. Hence, the identities of Lemma \ref{firstcase} show that
$$
p^*_r=(x^*_{r-1})^{\ell^*_{r-1}}(p^*_{r-1})^{(\ell')^*_{r-1}}=x_r^{\ell_r}p_r^{\ell_rh_{r-1}}p_r^{\ell'_r-\ell_rh_{r-1}}=p_{r+1}.
$$
Also, from $V_{r+1}=e(\mu)\mu(\phi)=V^*_r$, we deduce
$$
\begin{array}{l}
x^*_r=\op{H}_{\mu'}(\phi)(p^*_r)^{-V^*_r}=\op{H}_{\mu'}(\phi)p_{r+1}^{-V_{r+1}}=x_{r+1},\\
y^*_r=(x^*_r)^{e^*_r}(p^*_r)^{-h^*_r}=x_{r+1}^{e_{r+1}}p_{r+1}^{-h_{r+1}}=y_{r+1}.
\end{array}
$$
Thus, the generators of $\ggmp$ are the same for both MacLane chains of $\mu$.

On the other hand, $N^*_r=N_{\mu,\phi}=N_{r+1}$ depends only on $\mu,\phi$ by definition. In particular, for any nonzero $g\in K[x]$ we have
$$
s_{r+1}(g)=s^*_r(g),\qquad u_{r+1}(g)=u^*_r(g).
$$
This implies $R_{r+1}(g)=R^*_r(g)$ as in the proof of Lemma \ref{firstcase}. In fact, (\ref{mainR}) shows that
$$
x_{r+1}^{s_{r+1}(g)}p_{r+1}^{u_{r+1}(g)}R_{r+1}(g)(y_{r+1})=\op{H}_{\mu'}(g)=(x^*_r)^{s^*_r(g)}(p^*_r)^{u^*_r(g)}R^*_r(g)(y^*_r),
$$   
so that $R^*_r(g)(y^*_r)=R_{r+1}(g)(y_{r+1})=R_{r+1}(g)(y^*_r)$, which implies $R^*_r(g)=R_{r+1}(g)$ by the transcendence of $y^*_r$.
\end{proof}

\begin{definition}\label{optimal}\mbox{\null}

A MacLane chain of length $r$ is  \emph{optimal} if $\deg\phi_1<\cdots<\deg\phi_r$.
\end{definition}

By an iterative application of Lemma \ref{augmentation}, we may convert any MacLane chain of $\mu$ into an optimal MacLane chain. In fact, whenever we find an augmentation step with $\deg\phi_{i-1}=\deg\phi_i$, we may collapse this step to get a shorter MacLane chain. Let us call this ``shrinking" procedure an \emph{optimization step}. 

In an optimization step, all data of levels $0,\,1,\,\dots,\,i-1$ of the MacLane chain remain unchanged; 
the data of level $i-1$ are lost and the data of the $i$-th level change as indicated in Lemma \ref{firstcase}. By Corollary \ref{firstproof}, the data of levels $i+1,\dots,r$ remain unchanged too.

Let us now go back to the general situation of Theorem \ref{laststep}. We have a MacLane chain of length $r$ of $\mu$ such that  $\phi\nmid_\mu\phi_r$ and we extend it to a MacLane chain (\ref{extchain}) of the  augmented valuation $\mu'=[\mu;\phi,\nu]$. By applying a finite number of optimization steps to the MacLane chain of $\mu$, we may convert it into an optimal MacLane chain
$$
\mu^*_0\ \stackrel{\phi^*_1,\nu^*_1}\lra\  \mu^*_1\ \stackrel{\phi^*_2,\nu^*_2}\lra\ \cdots\  
\ \stackrel{\phi^*_{r^*},\nu^*_{r^*}}\lra\ \mu^*_{r^*}=\mu
$$
Since the polynomial $\phi^*_{r^*}=\phi_r$ remains unchanged, we may extend this chain as well to a MacLane chain of $\mu'$:
$$
\mu^*_0\ \stackrel{\phi^*_1,\nu^*_1}\lra\  \mu^*_1\ \stackrel{\phi^*_2,\nu^*_2}\lra\ \cdots\  
\ \stackrel{\phi^*_{r^*},\nu^*_{r^*}}\lra\ \mu^*_{r^*}=\mu\ \stackrel{\phi,\nu}\lra\ \mu_{r^*+1}=\mu' 
$$ 
By an iterative application of Corollary \ref{firstproof}, all data and operators attached to $\mu'$ by this extension of an optimal chain coincide with the data and operators attached to $\mu'$ through the original extended chain (\ref{extchain}).   
Therefore, in order to prove Theorem \ref{laststep}, we need only to compare the data attached to $\mu'$ through the MacLane chains obtained by extending two different optimal MacLane chains of $\mu$.   

Now, two optimal MacLane chains of the same valuation $\mu$ have the same length $r$, the same intermediate valuations $\mu_1,\dots,\mu_{r-1}$ and the same slopes $\nu_1,\dots,\nu_r$ \cite[Prop. 3.6]{Rid}.
Also, by Lemma \ref{unicity},
two families $\phi_1,\dots,\phi_r$ and $\phi^*_1,\dots,\phi^*_r$ are the key polynomials of two optimal MacLane chains of $\mu$ if and only if 
\begin{equation}\label{optimicity}
\deg \phi_i=\deg \phi^*_i,\quad\ \mu_i(\phi_i-\phi^*_i)\ge\mu_i(\phi_i),\quad 1\le i\le r.
\end{equation}
These polynomials satisfy $\phi^*_i\sim_{\mu_{i-1}}\phi_i$, but not necessarily $\phi^*_i\sim_{\mu_i}\phi_i$.
\bigskip

\noindent{\bf Proof of Theorem \ref{laststep} }
As mentioned above, we may assume that we deal with two MacLane chains of $\mu'$ which have been obtained by adding the augmentation step $\mu'=[\mu;\phi,\nu]$ to two optimal MacLane chains of $\mu$: 
$$
\as{.6}
\mu_0\ \begin{array}{c}\phi_1,\nu_1\\\lra\\\lra\\\phi^*_1,\nu_1\end{array}\  \mu_1\ \begin{array}{c}\phi_2,\nu_2\\\lra\\\lra\\\phi^*_2,\nu_2\end{array}\ \cdots
\ \begin{array}{c}\lra\\\lra\end{array}\ \mu_{r-1} 
\ \begin{array}{c}\phi_r,\nu_r\\\lra\\\lra\\\phi^*_r,\nu_r\end{array}\ \mu_r=\mu\ \stackrel{\phi,\nu}\lra \ \mu_{r+1}=\mu'
$$
The key polynomials of both MacLane chains satisfy (\ref{optimicity}). By hypothesis, $\phi\nmid_\mu\phi_r$ and $\phi\nmid_\mu\phi^*_r$. As usual, we mark with a superscript $(\ )^*$ all data and operators attached to the lower MacLane chain.

Note that $V_{r+1}=e(\mu)\mu(\phi)=V^*_{r+1}$. Hence, the numerical data
$$
h_i,\ e_i,\ \nu_i,\ V_i,\ \ell_i,\ \ell'_i,\ \quad 0\le i\le r+1
$$
coincide for both chains. By \cite[Lem. 4.13]{Rid}, we have
$$
\begin{array}{ll}
p^*_i=p_i,&\mbox{ for all } 1\le i\le r,\\
x^*_i=x_i,&\mbox{ for all } 1\le i\le r \mbox{ such that }e_i>1.
\end{array}
$$
Now, if $e_r>1$ we have $x^*_r=x_r$ and 
$$
p^*_{r+1}=(x^*_r)^{\ell_r}(p^*_r)^{\ell'_r}=x_r^{\ell_r}p_r^{\ell'_r}=p_{r+1}.
$$
If $e_r=1$ we have $\ell_r=0$, $\ell'_r=1$ and this leads to the same conclusion:
$$
p^*_{r+1}=(x^*_r)^{\ell_r}(p^*_r)^{\ell'_r}=p^*_r=p_r=x_r^{\ell_r}p_r^{\ell'_r}=p_{r+1}.
$$
As a consequence,
$$
\begin{array}{l}
x^*_{r+1}=\op{H}_{\mu'}(\phi)(p^*_{r+1})^{-V^*_{r+1}}=\op{H}_{\mu'}(\phi)(p_{r+1})^{-V_{r+1}}=x_{r+1},\\
y^*_{r+1}=(x^*_{r+1})^{e^*_{r+1}}(p^*_{r+1})^{-h^*_{r+1}}=(x_{r+1})^{e_{r+1}}(p_{r+1})^{-h_{r+1}}=y_{r+1}.
\end{array}
$$

By the very definition, $N_{r+1}=N_{\mu,\phi}=N^*_{r+1}$ depends only on $\mu$ and $\phi$. In particular, $s_{r+1}(g)=s^*_{r+1}(g)$, $u_{r+1}(g)=u^*_{r+1}(g)$,
for any nonzero $g\in K[x]$. This leads to $R_{r+1}=R^*_{r+1}$ by the usual argument using (\ref{mainR}) and the transcendence of $y_r$ over $\F_r$. 
\hfill{$\Box$}

\subsection{Variation of the data attached to one level}
Consider a fixed MacLane chain of $\mu$ of length $r$, as in (\ref{depth}). Once we know that the data and operators attached to the $r$-th level do not depend on the previous levels, our second aim is to analyze the variation of these data and operators when the key polynomial $\phi_r$ of that level changes.

By Lemma \ref{unicity}, the only way to obtain $\mu$ as an augmentation of $\mu_{r-1}$ is by taking $\mu=[\mu_{r-1};\phi^*_r,\nu_r]$, with $\phi^*_r=\phi_r+a$ such that $\deg a<\deg \phi_r$ and $\mu(a)\ge \mu(\phi_r)$. Since $\phi^*_r\sim_{\mu_{r-1}}\phi_r$, we have $\phi^*_r\nmid_{\mu_{r-1}}\phi_{r-1}$ too, so that it makes sense to consider another MacLane chain of $\mu$ as in (\ref{depth}), just by replacing $\phi_r$ by $\phi^*_r$.

As mentioned in section \ref{subsecML}, $\Gamma(\mu)$ is the subgroup of $\Q$ generated by $\Gamma(\mu_{r-1})$ and $\nu_r$; on the other hand, $e_r=e(\mu)/e( \mu_{r-1})$ is the least positive integer such that $e_r\Gamma(\mu)\subset\Gamma(\mu_{r-1})$. Hence, $\nu_r$ belongs to $\Gamma(\mu_{r-1})$ if and only if $e_r=1$. Hence, if $e_r>1$, then $\mu(\phi_r)=\mu_{r-1}(\phi_r)+\nu_r$ does not belong to $\Gamma(\mu_{r-1})$, and the equality $\mu(a)=\mu(\phi_r)$ cannot occur, because $\mu(a)=\mu_{r-1}(a)$ belongs to $\Gamma(\mu_{r-1})$. In other words,
\begin{equation}\label{e>1}
e_r>1\imp \mu(a)>\mu(\phi_r)\imp \phi^*_r\smu\phi_r.
\end{equation}

\begin{theorem}\label{lastlevel}
Consider two MacLane chains of an inductive valuation $\mu$, which differ only in the last augmentation step:
$$
\as{.6}
\mu_0\ \stackrel{\phi_1,\nu_1}\lra\  \mu_1\ \stackrel{\phi_2,\nu_2}\lra\ \cdots
\ \lra\ \mu_{r-2} 
\ \stackrel{\phi_{r-1},\nu_{r-1}}\lra\ \mu_{r-1}\ \begin{array}{c}\phi_r,\nu_r\\\lra\\\lra\\\phi^*_r,\nu_r\end{array}\;\mu_r=\mu
$$
Let us mark with a superscript $(\ )^*$ all data and operators attached to the lower MacLane chain. If $\phi^*_r\smu\phi_r$, we have

$$
p^*_r=p_r,\quad x^*_r=x_r,\quad y^*_r=y_r,\quad S^*_{\nu_r}=S_{\nu_r},\quad R^*_r=R_r.
$$

Assume that  $\phi^*_r\not\smu\phi_r$ and let $\eta:=R_r(\phi^*_r-\phi_r)\in\F_{\mu}^*$. Then,
$$
p^*_r=p_r,\quad x^*_r=x_r+p_r^{h_r}\eta,\quad y^*_r=y_r+\eta.
$$Further, for any nonzero $g\in K[x]$ let $s:=\ord_{y+\eta}R_r(g)$ and denote $P(g):=R_r(g)/(y+\eta)^s$. Then,
$$
s^*_r(g)=s,\qquad R_r^*(g)(y)=(y-\eta)^{s_r(g)}P(g)(y-\eta).
$$
\end{theorem}

\begin{proof}
By Lemma \ref{unicity},  $\phi_r^*=\phi_r+a$ with $\deg a<\deg\phi$ and $\mu(a)\ge \mu(\phi)$. 

All data attached to levels $i<r$ coincide for the both chains. Therefore,
$$
\begin{array}{l}
V_r=e_{r-1}f_{r-1}(e_{r-1}V_{r-1}+h_{r-1})=V^*_r,\\
p_r=x_{r-1}^{\ell_{r-1}}p_{r-1}^{\ell'_{r-1}}=p^*_r.
\end{array}
$$
Also, since $\nu_r=\nu^*_r$, we have $h^*_r=h_r$ and $e^*_r=e_r$.

Suppose $\phi^*_r\smu\phi_r$. Then,
$$
\begin{array}{l}
x^*_r=\hm(\phi^*_r)(p^*_r)^{-V^*_r}=\hm(\phi_r)p_r^{-V_r}=x_r,\\
y^*_r=(x^*_r)^{e^*_r}(p^*_r)^{-h^*_r}=x_r^{e_r}p_r^{-h_r}=y_r.
\end{array}
$$

Now, consider a nonzero $g\in K[x]$, and let $S_{\nu_r}(g)$, $S^*_{\nu_r}(g)$ be the $\nu_r$-components of $N_r(g)$, $N^*_r(g)$, respectively. Both segments lie on the line of slope $-\nu_r$ cutting the vertical axis at the point $(0,\mu(g))$ (see Figure \ref{figComponent}). Hence, in order to check that $S_{\nu_r}(g)=S^*_{\nu_r}(g)$ it suffices to show that the end points of both segments have the same abscissas. 

Let $\ord_{\mu,\phi_r}(g)$ be the largest integer $k$ such that $\phi_r^k\mid_\mu g$, namely the order with which the prime $\hm(\phi_r)$ divides $\hm(g)$ in $\ggm$. By \cite[Lem. 2.6]{Rid}, the abscissas of the end points of $S_{\nu_r}(g)$ are:
$$
s_r(g)=\ord_{\mu,\phi_r}(g),\quad s'_r(g)=\ord_{\mu',\phi_r}(g),
$$
where $\mu'=[\mu_{r-1};\phi_r,\nu_r-\epsilon]$ for a suficiently small positive rational number $\epsilon$. Since $\phi^*_r\smu\phi_r$, we have $\mu(a)>\mu(\phi_r)$, so that $\mu'(a)=\mu_{r-1}(a)=\mu(a)>\mu(\phi_r)>\mu'(\phi_r)$, and we have $\phi^*_r\sim_{\mu'}\phi_r$ as well. Hence,
$$
\as{1.2}
\begin{array}{c}
s_r(g)=\ord_{\mu,\phi_r}(g)=\ord_{\mu,\phi^*_r}(g)=s^*_r(g),\\ s'_r(g)=\ord_{\mu',\phi_r}(g)=\ord_{\mu',\phi^*_r}(g)=(s')^*_r(g).
\end{array}
$$
This implies  $S_{\nu_r}(g)=S^*_{\nu_r}(g)$.

In particular, $u_r(g)=u^*_r(g)$. We may now deduce $R_r(g)=R^*_r(g)$ by the usual argument using (\ref{mainR}) and the transcendence of $y_r$ over $\F_r$. This ends the proof of the theorem in the case $\phi^*_r\smu\phi_r$.

Suppose now $\phi^*_r\not\smu\phi_r$, or equivalently $\mu(a)=\mu(\phi)$, which implies $e_r=1$ by (\ref{e>1}). Both Newton polygons $N_r(a)=N^*_r(a)$ coincide with the point $(0,\mu(a))=(0,\mu(\phi_r))=(0,(V_r+h_r)/e(\mu_{r-1}))$. Hence,
$$
s_r(a)=0=s^*_r(a),\quad u_r(a)=V_r+h_r=u^*_r(a).
$$
By (\ref{mainR}), we have
$$
(p^*_r)^{V_r+h_r}R^*_r(a)=\hm(a)=(p_r)^{V_r+h_r}R_r(a),
$$
which implies $\eta:=R_r(a)=R^*_r(a)=\hm(a)p_r^{-V_r-h_r}$, since $p^*_r=p_r$. Thus,
$$
x^*_r=\hm(\phi^*_r)p_r^{-V_r}=\left(\hm(\phi_r)+\hm(a)\right)p_r^{-V_r}=x_r+p^{h_r}\eta,
$$
leading to $y^*_r=x^*_r(p^*_r)^{-h^*_r}=y_r+\eta$.

Now, for a nonzero $g\in K[x]$, denote $\alpha=\mu(g)$ and $u_r(\alpha)=e(\mu)\alpha\in\Z$. Consider the polynomials $$R_{r,\alpha}(g)=y^{s_r(g)}R_r(g),\quad  R^*_{r,\alpha}(g)=y^{s^*_r(g)}R^*_r(g).$$ By \cite[Thm. 4.1]{Rid}, we have identities:
$$
(p_r^*)^{u_r(\alpha)}R^*_{r,\alpha}(g)(y^*_r)=\hm(g)=
(p_r)^{u_r(\alpha)}R_{r,\alpha}(g)(y_r).
$$
Since $p^*_r=p_r$, we deduce:
$$
R_{r,\alpha}(g)(y_r)= R^*_{r,\alpha}(g)(y^*_r)=R^*_{r,\alpha}(g)(y_r+\eta),
$$
which implies $R_{r,\alpha}(g)(y)= R^*_{r,\alpha}(g)(y+\eta)$, by the transcendence of $y_r$ over $\F_r$. Let us rewrite this equality in terms of the original residual polynomials:
\begin{equation}\label{Ralpha}
y^{s_r(g)}R_r(g)(y)=(y+\eta)^{s^*_r(g)}R_r^*(g)(y+\eta).
\end{equation}
Since $r>0$, we have $y\nmid R_r(g)$, $y\nmid R^*_r(g)$ (cf. section \ref{subsecML}). Hence, $(y+\eta)\nmid R^*_r(g)(y+\eta)$, and the equality (\ref{Ralpha}) shows that $s^*_r(g)=\ord_{y+\eta}R_r(g)$ and $(y-\eta)^{s_r(g)}P(g)(y-\eta)=R_r^*(g)(y)$. 
\end{proof}

\section{Equivalence of types}\label{secTypes}
Types are computational representations of certain mathematical objects. It is natural to consider two types to be equivalent when they represent the same objects. In sections \ref{subsecNormVal}, \ref{subsecTypes}, we recall the objects parameterized by types and in section \ref{subsecTypesEquiv} we characterize the equivalence of types in terms of checkable conditions on the data supported by them (Lemma \ref{optstep} and Proposition \ref{charequiv}) and in terms of other invariants (Theorem \ref{finalchar}).

\subsection{Normalized inductive valuations}\label{subsecNormVal}
In a computational context, it is natural to normalize inductive valuations in order to get groups of values equal to $\Z$.

Given a MacLane chain of an inductive valuation $\mu$:
$$
\mu_0\ \stackrel{\phi_1,\nu_1}\lra\  \mu_1\ \stackrel{\phi_2,\nu_2}\lra\ \cdots
\ \stackrel{\phi_{r-1},\nu_{r-1}}\lra\ \mu_{r-1} 
\ \stackrel{\phi_{r},\nu_{r}}\lra\ \mu_{r}=\mu
$$
we consider the normalized valuations:
$$
v_i:=e(\mu_i)\mu_i=e_1\cdots e_i\,\mu_i,\quad 0\le i\le r,
$$
with group of values $\Gamma(v_i)=v_i(K(x)^*)=e(\mu_i)\Gamma(\mu_i)=\Z$. The property ${\mu_i}_{\mid K}=v$ translates into  ${v_i}_{\mid K}=e_1\cdots e_i\, v$.

The graded algebras $\gg(\mu_i)$ and $\gg(v_i)$ coincide up to the change of graduation given by the group isomorphism 
$$
\Gamma(\mu_i)\iso \Gamma(v_i)=\Z,\quad\ \alpha\mapsto e_1\cdots e_i\,\alpha.
$$    
The piece of degree zero $\Delta_i:=\Delta(\mu_i)=\Delta(v_i)$ is the same for both valuations. Further, for any $g,h\in K[x]$ we obviously have
$$
g\mid_{\mu_i}h \sii g\mid_{v_i}h,\qquad   
g\sim_{\mu_i}h \sii g\sim_{v_i}h.
$$

Also, consider the normalized slopes
$$
\la_i:=e(\mu_{i-1})\nu_i=h_i/e_i,\quad 0\le i\le r.
$$
The augmentation step $\mu_i=[\mu_{i-1};\phi_i,\nu_i]$ translates into  $v_i=[e_iv_{i-1};\phi_i,\la_i]$. If $g=\sum_{0\le s}a_s\phi_i^s$ is the $\phi_i$-expansion of a nonzero $g\in K[x]$, we have
$$
v_i(g)=\mn\left\{e_iv_{i-1}(a_s\phi_i^s)+s\la_i\mid 0\le s\right\}=\mn\left\{v_i(a_s\phi_i^s)\mid 0\le s\right\}.
$$  
The property $\mu_{i-1}<\mu_i$ translates into $e_iv_{i-1}<v_i$.

With the obvious definition, we get a MacLane chain of $v_r$:
$$
\mu_0=v_0\ \stackrel{\phi_1,\la_1}\lra\  v_1\ \stackrel{\phi_2,\la_2}\lra\ \cdots
\ \stackrel{\phi_{r-1},\la_{r-1}}\lra\ v_{r-1} 
\ \stackrel{\phi_{r},\la_r}\lra\ v_r
$$
with attached data and operators as described in section \ref{subsecML}. 

This approach has the advantage that the Newton polygons $N_{v_{i-1},\phi_i}(g)$ are derived from clouds of points in $\R^2$ with integer coordinates. The affinity $\hh(x,y)=(x,e_1\cdots e_{i-1}\,y)$ maps  $N_{\mu_{i-1},\phi_i}(g)$ to  $N_{v_{i-1},\phi_i}(g)$.
This affinity maps a side of slope $\rho$ to a side of slope $e_1\cdots e_{i-1}\,\rho$ with the same abscissas of the end points. Thus, the role of the $\nu_i$-component is undertaken by the $\la_i$-component in the normalized context. 
Note that the left end point of the $\nu_i$-component of $N_{\mu_{i-1},\phi_i}(g)$ is $(s_i(g),u_i(g)/e_1\cdots e_{i-1})$, while the left end point of the $\la_i$-component of $N_{v_{i-1},\phi_i}(g)$ is simply $(s_i(g),u_i(g))$.

The rest of data and operators attached to both MacLane chains coincide. Specially, for $0\le i<r$, we have the same residual polynomial operators:
$$
R_{\mu_{i-1},\phi_i,\nu_i}=R_i=R_{v_{i-1},\phi_i,\la_i}\colon K[x]\lra \F_i[y],
$$ 
and the same family of prime polynomials $\psi_i=R_i(\phi_{i+1})\in\F_i[y]$.

\subsection{Types over $(K,v)$}\label{subsecTypes}
A type of order $r$ is a collection of objects, distributed into levels:
$$\ty=(\varphi_0;(\phi_1,\la_1,\varphi_1);\dots;(\phi_r,\la_r,\varphi_r)),
$$ such that the pairs $\phi_i,\la_i$ determine a McLane chain of a normalized inductive valuation $v_\ty$:
\begin{equation}\label{depth2}
v_0\ \stackrel{\phi_1,\la_1}\lra\  v_1\ \stackrel{\phi_2,\la_2}\lra\ \cdots
\ \stackrel{\phi_{r-1},\la_{r-1}}\lra\ v_{r-1}\ \stackrel{\phi_r,\la_r}\lra\ v_r=v_\ty
\end{equation}
and the data $\varphi_0,\dots,\varphi_{r}$ build a tower of finite field extensions of $\F$:
$$
\F_{0,\ty}:=\F\,\lra\, \F_{1,\ty}\,\lra\,\cdots\,\lra\, \F_{r,\ty}\,\lra\, \F_{r+1,\ty}
$$
constructed as follows. Each $\varphi_i\in\F_{i,\ty}[y]$ is a monic irreducible polynomial, such that $\varphi_i\ne y$ for $i>0$. The field $\F_{i+1,\ty}$ is defined to be $\F_{i,\ty}[y]/(\varphi_i)$.

Also, there is an specific procedure to compute certain residual polynomial operators
$$
R_{i,\ty}\colon K[x]\lra \F_{i,\ty}[y],\quad 0\le i\le r, 
$$ 
such that $\varphi_i=R_{i,\ty}(\phi_{i+1})$ for $0\le i<r$. The essential fact is that these objects reproduce the tower $\F_0\,\to\,\cdots \,\to\,\F_r$ and the residual polynomial operators $R_i$ attached to the MacLane chain of $v_\ty$. More precisely, there is a commutative diagram of vertical isomorphisms
$$
\as{1.2}
\begin{array}{ccccccc}
\F=\F_{0,\ty}\ &\subset& \F_{1,\ty}\ &\subset&\cdots &\subset &\F_{r,\ty}\ \\  
\quad\ \,\|\iota_0 &&\ \downarrow\iota_1&&\cdots&&\ \downarrow\iota_r\\
\F=\F_0\quad &\subset& \F_1\ &\subset&\cdots &\subset &\F_r\
\end{array}
$$
such that $R_i=\iota_{i}[y]\circ R_{i,\ty}$ for all $0\le i\le r$. In particular, $$\psi_i=R_i(\phi_{i+1})=\iota_i[y]\left(R_{i,\ty}(\phi_{i+1})\right)=\iota_i[y]\left(\varphi_i\right),\quad 0\le i<r.
$$

The isomorphisms $\iota_0,\dots,\iota_r$ are uniquely determined by the isomorphisms $j_0,\dots,j_r$ defined in (\ref{ji}). In fact, the isomorphism $\iota_0$ is the identity map on $\F_{0,\ty}=\F=\F_0$, while $\iota_{i+1}$ is determined by the following commutative diagram of vertical isomorphisms: 
\begin{center}
\setlength{\unitlength}{4mm}
\begin{picture}(18,8)
\put(0,3.5){$\as{1.2}
\begin{array}{ccl} 
\F_{i,\ty}[y]&\twoheadrightarrow&\F_{i,\ty}[y]/(\varphi_i)\ =\ \F_{i+1,\ty}\\
\!\!\!\!\!\iota_i[y]\downarrow\hphantom{m}&&\qquad\downarrow\\
\F_i[y]&\twoheadrightarrow&\ \F_i[y]/(\psi_i)\\
j_i\downarrow\hphantom{m}&&\qquad\downarrow\\
\Delta_i&\twoheadrightarrow&\quad \Delta_i/\ll_i\,\iso\F_{i+1}\subset \Delta_{i+1}
\end{array}
$}
\put(13.2,5.5){\vector(0,-1){3.4}}
\put(13.8,3.6){$\iota_{i+1}$}
\end{picture}
\end{center}

Therefore, for the theoretical considerations of this paper it will be harmless to consider the isomorphisms $\iota_0,\dots,\iota_r$ as identities. That is,  we shall identify all data and operators supported by $\ty$ with the analogous data and operators attached to $v_\ty$: 
$$\F_{i}=\F_{i,\ty},\quad R_i=R_{i,\ty},\quad 0\le i\le r.$$
In particular, $\psi_i=\varphi_i$ for $0\le i<r$. 
According to this convention, from now on a type will be a collection of objects:
$$\ty=(\psi_0;(\phi_1,\la_1,\psi_1);\dots;(\phi_r,\la_r,\psi_r)),
$$ such that the pairs $\phi_i,\la_i$ determine a McLane chain of a normalized inductive valuation $v_\ty$ as in (\ref{depth2}), and $\psi_i\in\F_i[y]$ are the monic irreducible polynomials determined by the MacLane chain too, for $0\le i<r$. 

What is the role of the prime polynomial $\psi_r\in\F_r[y]$? Let us denote by $$\mu_\ty:=(e_1\cdots e_r)^{-1}v_\ty,\qquad f_r:=\deg\psi_r,$$
the corresponding non-normalized inductive valuation attached to $\ty$ and the degree of $\psi_r$, respectively. Thanks to the isomorphism $j_r$, the polynomial $\psi_r$
 determines a maximal ideal of $\Delta_r=\Delta(\mu_\ty)$:
$$
\ll_\ty:=j_r\left(\psi_r\F_r[y]\right)=\psi_r(y_r)\Delta(\mu_\ty)\in\mx(\Delta(\mu_\ty)).
$$ 
The pair $(\mu_\ty,\ll_\ty)$, or equivalently $(v_\ty,\ll_\ty)$, is the ``raison d'\^{e}tre" of $\ty$.

\subsection{Representatives of types}\label{subsecRepr}

Denote $\mu:=\mu_\ty$, $\Delta:=\Delta(\mu)$, $\ll:=\ll_\ty$. 
The maximal ideal $\ll$ determines a certain subset of key polynomials for $\mu$, which are called \emph{representatives} of the type $\ty$. By definition, the set $\rep(\ty)$ of all representatives of $\ty$ is:
$$
\rep(\ty)=\left\{\phi\in \kpm\mid \rr_\mu(\phi)=\ll_\ty \right\}\subset\kpm.
$$ 
Since the residual ideal map $\rr_\mu\colon \kpm\to\mx(\Delta)$ is onto \cite[Thm. 5.7]{Rid}, the set $\rep(\ty)$ is always non-empty. By (\ref{repr}), the representatives of $\ty$ constitute one of the $\mu$-equivalence classes of the set $\kpm$. 

For any monic $\phi\in K[x]$, the property of being a representative of the type $\ty$ is characterized too by the following properties \cite[Lem. 3.1]{gen}:
\begin{equation}\label{charrep}
\phi\in\rep(\ty)\sii\phi\in\oo[x],\ \deg\phi=e_rf_r m_r, \ R_r(\phi)=\psi_r.  
\end{equation}
By (\ref{charrep}) and (\ref{em}), the representatives of $\ty$ are proper key polynomials for $\mu$. 

Let $\phi$ be any representative of a type $\ty$ of order $r>0$. By \cite[Cor. 5.3]{Rid}, $R_\mu(\phi_r)=y_r\Delta\ne\psi_r(y_r)\Delta=\ll_\ty$, because $\psi_r\ne y$. By (\ref{repr}), $\phi\nmid_\mu\phi_r$, and we may extend the MacLane chain of $ \mu$
to a MacLane chain of length $r+1$ of the augmented valuation $\mu'=[\mu;\phi,\nu]$, where $\nu$ is an arbitray positive rational number. By choosing an arbitrary monic irreducible polynomial $\psi\in\F_{r+1}[y]=\F_{\mu'}[y]$, we construct a type of order $r+1$ extending $\ty$:
$$
\ty'=(\ty;(\phi,\nu,\psi)):=(\psi_0;(\phi_1,\la_1,\psi_1);\dots;(\phi_r,\la_r,\psi_r);(\phi,\la,\psi)),
$$  
where $\la=e(\mu)\nu$.

\begin{definition}
Let $\ty$ be a type of order $r\ge0$. For any $g\in K[x]$ we define $\ord_\ty(g):=\ord_{\psi_{r}}R_{r}(g)$; that is, the greatest integer $a$ such that $\psi_r^a$ divides $R_{r}(g)$ in $\F_r[y]$.
\end{definition}

Since the operators $R_i$ are multiplicative \cite[Cor. 4.11]{Rid}, the identity $\ord_\ty(gh)=\ord_\ty(g)+\ord_\ty(h)$ holds for all $g,h\in K[x]$.

\subsection{Equivalence of types}\label{subsecTypesEquiv}
Let $\ty$ be a type of order $r$ with representative $\phi$, and let $\la\in\Q_{>0}$. We denote
$$
\begin{array}{l}
N_i:=N_{v_{i-1},\phi_i},\quad 1\le i\le r;\qquad 
N_{\ty,\phi}:=N_{v_\ty,\phi},\\
R_{\ty,\phi,\la}:=R_{v_\ty,\phi,\la}=R_{\mu_\ty,\phi,\la/e(\mu_\ty)}.
\end{array}
$$
Note that $R_{\ty,\phi,\la}$ is well-defined by Theorem \ref{laststep}.

\begin{definition}\label{defequiv}
Two types $\ty$, $\ty^*$ are equivalent if $v_\ty=v_{\ty^*}$ and  $\ll_\ty=\ll_{\ty^*}$. In this case we write $\ty\equiv\ty^*$.
\end{definition}

The next result is an immediate consequence of the definitions.

\begin{proposition}\label{samereps}
Let $\ty$, $\ty^*$ be two equivalent types. Then,
\begin{enumerate}
\item $\rep(\ty^*)=\rep(\ty)$.
\item For any $\phi\in\rep(\ty)$ and any $\la\in\Q_{>0}$, we have
$N_{\ty^*,\phi}=N_{\ty,\phi}$ and $R_{\ty^*,\phi,\la}=R_{\ty,\phi,\la}$.
\end{enumerate}
\end{proposition}

The order of a type is not preserved by equivalence. 
In order to find a characterization of the equivalence of types in terms of the data supported by them, we consider \emph{optimization steps} derived from the optimization steps for MacLane chains. 

\begin{definition}\label{stationary}
Let $\ty$ be a type of order $r$. We say that a level $(\phi_i,\la_i,\psi_i)$ of $\ty$ is \emph{stationary} if $e_i=f_i=1$, or equivalently, if $\la_i\in\Z$ and $\deg\psi_i=1$. 

We say that $\ty$ is \emph{optimal} if $\deg\phi_1<\cdots <\deg\phi_r$, or equivalently, if all levels $i<r$ are non-stationary. We say that $\ty$ is \emph{strongly optimal} if all levels $i\le r$ are non-stationary. 
\end{definition}

\begin{lemma}\label{optstep}
For $r\ge 2$, let $\ty_0$ be a type of order $r-2$. Consider a type 
$$
\ty=(\ty_0;(\phi_{r-1},\la_{r-1},\psi_{r-1});(\phi_r,\la_r,\psi_r))
$$
of order $r$ whose $(r-1)$-th level is stationary. Then, $\phi_r$ is a representative of $\ty_0$ and the type $\ty^*=(\ty_0;(\phi_r,\la_{r-1}+\la_r,\psi_r))$
is equivalent to $\ty$. Moreover, 
\begin{equation}\label{NR}
N_{r-1}^*=\hh\circ N_r,\qquad R_{r-1}^*=R_r,
\end{equation}
where $\hh$ is the affinity $\hh(x,y)=(x,y-\la_{r-1}x)$.
Thus, $\ord_\ty=\ord_{\ty^*}$ as functions on $K[x]$.
\end{lemma}

\begin{proof}
By Lemma \ref{augmentation}, $\phi_r$ is a key polynomial for $\mu_{\ty_0}$ and
$$
\mu_\ty=[\mu_{\ty_0};\phi_r,\nu_{r-1}+\nu_r]=\mu_{\ty^*},
$$ 
where $\nu_{r-1}=\la_{r-1}/e_1\cdots e_{r-2}$ and $\nu_r=\la_r/e_1\cdots e_{r-1}=\la_r/e_1\cdots e_{r-2}$.

By \cite[Lem. 5.2]{Rid}, $N_{r-1}(\phi_r)$ is one-sided of negative slope $-\nu_{r-1}$; hence, \cite[Lem. 2.1]{Rid}
shows that $\phi_{r-1}\mid_{\mu_{\ty_0}}\phi_r$. By (\ref{repr}), we have 
$$
\ll_{\ty_0}=\rr_{\mu_{\ty_0}}(\phi_{r-1})=\rr_{\mu_{\ty_0}}(\phi_r),
$$
so that $\phi_r$ is a representative of $\ty_0$. 

The identities (\ref{NR}) are a consequence of Lemma \ref{firstcase}. 
Finally, let $\phi$ be a representative of $\ty$, so that $\ll_\ty=\rr_{\mu_\ty}(\phi)$. 
Since $R_r(\phi)=\psi_r$, we deduce that $R^*_{r-1}(\phi)=R_r(\phi)=\psi_r$. Hence, $\phi$
is a representative of $\ty^*$ too, because it satisfies the conditions of (\ref{charrep}), characterizing the representatives of a type. Therefore,
$$
\ll_{\ty^*}=\rr_{\mu_{\ty^*}}(\phi)=\rr_{\mu_{\ty}}(\phi)=\ll_\ty,
$$ 
and the types $\ty$, $\ty^*$ are equivalent.
\end{proof}

After a finite number of these optimization steps we may convert any type into an optimal type in the same equivalence class. Thus, in order to check if two types are equivalent we need only to characterize the equivalence of optimal types. The characterization we obtain is an immediate consequence of the characterization of MacLane optimal chains \cite[Prop. 3.6]{Rid} and Lemma \ref{lastlevel}.

\begin{proposition}\label{charequiv}
Two optimal types
$$
\as{1.2}
\begin{array}{l}
\ty=(\psi_0;(\phi_1,\la_1,\psi_1);\dots;(\phi_r,\la_r,\psi_r)),\\
\ty^*=(\psi^*_0;(\phi^*_1,\la^*_1,\psi^*_1);\dots;(\phi^*_{r^*},\la^*_{r^*},\psi^*_{r^*})).
\end{array}
$$
are equivalent if and only if they satisfy the following conditions:
\begin{itemize}
\item $r=r^*$.
\item $\la_i=\la^*_i$ for all $1\le i\le r$. 
\item $\deg\phi_i=\deg\phi^*_i$ and $\mu_i(a_i)\ge \mu_i(\phi_i)$ for all $1\le i\le r$, where $a_i:=\phi^*_i-\phi_i$.
\item $\psi_r^*(y)=\psi_r(y-\eta_r)$, where $\eta_0:=0$ and for all $1\le i\le r$ we take
$$
\eta_i:=
\begin{cases}
0,&  \mbox{ if }\mu_i(a_i)>\mu_i(\phi_i) \quad\mbox{ (i.e. }\phi^*_i\sim_{\mu_i}\phi_i),\\
R_i(a_i)\in\F_i^*,&  \mbox{ if }\mu_i(a_i)=\mu_i(\phi_i) \quad\mbox{ (i.e. }\phi^*_i\not\sim_{\mu_i}\phi_i).
\end{cases}
$$
\end{itemize}

In this case,  $\psi_i^*(y)=\psi_i(y-\eta_i)$ for all $0\le i\le r$, and for any nonzero $g\in K[x]$ we have:
$$
s^*_i(g)=\ord_{y+\eta_i}R_i(g),\quad R^*_i(g)(y)=(y-\eta_i)^{s_i(g)}P_i(g)(y-\eta_i),\quad  1\le i\le r,
$$ 
where $P_i(g)(y):=R_i(g)(y)/(y+\eta_i)^{s^*_i(g)}$.
\end{proposition}

We may derive from this ``practical" characterization of the equivalence of types some more conceptual characterizations.

\begin{theorem}\label{finalchar}
For any pair of types $\ty$, $\ty^*$, the following conditions are equivalent.

\begin{enumerate}
\item $\ty\equiv\ty^*$
\item $\ord_\ty=\ord_{\ty^*}$
\item $\rep(\ty)=\rep(\ty^*)$
\end{enumerate}
\end{theorem}

\begin{proof}
Let us prove that (1) implies (2).
By Lemma \ref{optstep}, the function $\ord_\ty$ is preserved by the optimization steps. Hence, we may assume that the types are optimal.  

Take $g\in K[x]$ a nonzero polynomial.
For two equivalent types of order $r=0$ we have $R_0=R^*_0$ and $\psi_0=\psi^*_0$; thus, 
$$\ord_\ty(g)=\ord_{\psi_0}(R_0(g))=\ord_{\psi^*_0}(R^*_0(g))=\ord_{\ty^*}(g).$$

If $r>0$, we have $\psi_r\ne y$ and $\psi^*_r\ne y$. By Proposition \ref{charequiv}, $\psi^*_r(y)=\psi_r(y-\eta_r)$, and this implies 
$\psi_r\ne y+\eta_r$, $\psi^*_r\ne y-\eta_r$. Hence,  Proposition \ref{charequiv} shows that
\begin{equation*}
\begin{split}
\ord_\ty(g)=&\ord_{\psi_r}R_r(g)=\ord_{\psi_r}P_r(g)=\ord_{\psi^*_r}P_r(g)(y-\eta_r)\\=&\ord_{\psi^*_r}R^*_r(g)=\ord_{\ty^*}(g).
\end{split}
\end{equation*}

On the other hand, (\ref{charrep}) characterizes the representatives of a type $\ty$ as monic polynomials $\phi\in\oo[x]$ with minimal degree satisfying $\ord_\ty(\phi)=1$; thus, (2) implies (3).

Finally, let us prove that (3) implies (1). Let us denote $\mu=\mu_\ty$, $\mu^*=\mu_{\ty^*}$. It suffices to show that $\mu=\mu^*$, because then any common representative $\phi\in\rep(\ty)\cap\rep(\ty^*)$ leads to $\ll_\ty=\rr_{\mu}(\phi)=\rr_{\mu^*}(\phi)=\ll_{\ty^*}$, so that $\ty$ and $\ty^*$ are equivalent.

Take $\phi\in\rep(\ty)\cap\rep(\ty^*)$ a common representative of $\ty$ and $\ty^*$. 
Let $\mu_{\infty}$ be the pseudo-valuation on $K[x]$ obtained as the composition:
$$
\mu_{\infty}\colon K[x] \hookrightarrow K_v[x]\lra K_\phi\stackrel{v}\lra \Q\cup\{\infty\},
$$the second mapping being determined by $x\mapsto \t$, a root of $\phi$ in $\overline{K}_v$. By \cite[Prop. 1.9]{Rid}, we have $\mu<\mu_\infty$,  $\mu^*<\mu_\infty$, and for any nonzero $g\in K[x]$:
\begin{equation}\label{a}
\mu(g)<\mu_\infty(g)\sii \phi\mmu g,\quad
\mu^*(g)<\mu_\infty(g)\sii \phi\mid_{\mu^*} g.
\end{equation}

Since the interval $[\mu_0,\mu_\infty]$ is totally ordered \cite[Thm. 7.5]{Rid}, after exchanging the role of $\mu$ and $\mu^*$ if necessary, we must have 
$$
\mu\le \mu^*<\mu_\infty.
$$ 
The proof will be complete if we show that the conditions $\mu<\mu^*<\mu_\infty$ and $\rep(\ty^*)=\rep(\ty)$ lead to a contradiction.

Let $\Phi_{\mu,\mu_\infty}$ be the set of all monic polynomials $\varphi\in K[x]$ of minimal degree satisfying $\mu(\varphi)<\mu_\infty(\varphi)$. Let $\deg \Phi_{\mu,\mu_\infty}$ be the common degree of all polynomials in this set.

We claim that $\phi$ belongs to $\Phi_{\mu,\mu_\infty}$. In fact, the inequality $\mu(\phi)<\mu_\infty(\phi)=\infty$ is obvious. On the other hand, for any $a\in K[x]$ of degree less than $\deg\phi$, the $\mu$-minimality of $\phi$ implies that $\phi\nmid_\mu a$; by (\ref{a}), we deduce that $\mu(a)=\mu_\infty(a)$. 

By Lemma \ref{inequality} below, there is a unique maximal ideal $\ll\in\mx(\Delta(\mu))$ such that  
$$
\Phi_{\mu,\mu_\infty}=\left\{\varphi \in\kpm\mid \rr_{\mu}(\varphi)=\ll\right\}.
$$
Since $\phi\in\Phi_{\mu,\mu_\infty}$ and $\rr_\mu(\phi)=\ll_\ty$, we see that
$$
\Phi_{\mu,\mu_\infty}=\left\{\varphi \in\kpm\mid \rr_{\mu}(\varphi)=\ll_\ty\right\}=\rep(\ty)=\left\{\varphi \in \kpm\mid \varphi\smu\phi\right\}.
$$
An analogous argument shows that
$$
\Phi_{\mu^*,\mu_\infty}=\rep(\ty^*)=\left\{\varphi \in \op{KP}(\mu^*)\mid \varphi\sim_{\mu^*}\phi\right\}.
$$

Also, Lemma \ref{inequality} shows that $\Phi_{\mu,\mu^*}$ is one of the $\mu$-equivalence classes in $\kpm$. Hence, if we show that  $\Phi_{\mu,\mu^*}\subset \Phi_{\mu,\mu_\infty}$, these two sets must coincide. 
In fact, a polynomial $\varphi\in\Phi_{\mu,\mu^*}$ is a key polynomial for $\mu$ with $\mu(\varphi)<\mu^*(\varphi)\le \mu_\infty(\varphi)$. By (\ref{a}), we have $\phi\mmu\varphi$, which implies $\rr_\mu(\phi)\supset\rr_\mu(\varphi)$; since $\rr_\mu(\phi)$, $\rr_\mu(\varphi)$ are maximal ideals of $\Delta(\mu)$, they coincide. Thus, $\phi\smu\varphi$, so that $\varphi$ belongs to $\Phi_{\mu,\mu_\infty}$. 

In particular, $\phi$ belongs to $\Phi_{\mu,\mu^*}=\Phi_{\mu,\mu_\infty}$. Consider the positive rational number $\nu=\mu^*(\phi)-\mu(\phi)$. By \cite[Thm. 1.15]{Vaq}, the augmented valuation 
$\mu'=[\mu;\phi,\nu]$ satisfies $\mu<\mu'\le\mu^*$ and $\mu'(\phi)=\mu(\phi)+\nu=\mu^*(\phi)$. 

We claim that $\mu'=\mu^*$. In fact, if $ \mu'<\mu^*$, then we could replace $\mu$ by $\mu'$ in the above arguments to deduce $\Phi_{\mu',\mu^*}=\Phi_{\mu',\mu_\infty}$. Therefore, 
$$
\deg\phi=\deg\Phi_{\mu,\mu^*}\le\deg\Phi_{\mu',\mu^*}=\deg\Phi_{\mu',\mu_\infty}\le\deg \Phi_{\mu^*,\mu_\infty}=\deg\phi.
$$ 
We deduce $\deg\Phi_{\mu,\mu^*}=\deg\Phi_{\mu',\mu^*}$, and this leads to $\Phi_{\mu,\mu^*}\supset\Phi_{\mu',\mu^*}$, because $\mu'(\varphi)<\mu^*(\varphi)$  implies obviously 
 $\mu(\varphi)<\mu^*(\varphi)$. Similarly, $\deg\Phi_{\mu',\mu_\infty}=\deg \Phi_{\mu^*,\mu_\infty}$, leading to $\Phi_{\mu',\mu_\infty}\supset\Phi_{\mu^*,\mu_\infty}$. Hence,
$$
\rep(\ty^*)=\Phi_{\mu^*,\mu_\infty}\subset\Phi_{\mu',\mu_\infty}=\Phi_{\mu',\mu^*}\subset \Phi_{\mu,\mu^*}=\Phi_{\mu,\mu_\infty}=\rep(\ty).
$$
The hypothesis $\rep(\ty)=\rep(\ty^*)$ implies $\Phi_{\mu',\mu^*}=\Phi_{\mu,\mu^*}$, which is impossible, because $\phi$ does not belong to $\Phi_{\mu',\mu^*}$. Therefore, $\mu^*=\mu'=[\mu;\phi,\nu]$. 

Since $\phi$ is a proper key polynomial for $\mu$, there exists a MacLane chain of $\mu^*$ such that $\phi,\nu$ are the augmentation data of the last level. Hence, $m_{\mu^*}= \deg \phi$ and $e_{\mu^*}$ is the least positive integer such that $e_{\mu^*} \nu\in\Gamma(\mu)$. Since $\phi$ is a proper key polynomial for $\mu^*$, (\ref{em}) shows that $\deg\phi\ge e_{\mu^*}m_{\mu^*}=e_{\mu^*}\deg\phi$. Thus, $e_{\mu^*}=1$, or equivalently, $\nu \in\Gamma(\mu)$.

By Lemma \ref{values} below, there exists $a\in K[x]$ of degree less than $\deg\phi$, such that
$\mu(a)=\mu(\phi)+\nu$. Take $\varphi=\phi+a$. Since $\varphi\smu \phi$ and $\deg\varphi=\deg\phi$, Lemma \ref{mid=sim} shows that $\varphi$ is a key polynomial for $\mu$, and $\varphi\in\rep(\ty)$ by (\ref{repr}). However, $\varphi\not\sim_{\mu^*}\phi$, because $\mu^*(a)=\mu(a)=\mu^*(\phi)$ is not greater than $\mu^*(\phi)$. Hence, $\varphi \not\in\rep(\ty^*)$, and this contradicts our hypothesis.
\end{proof}

We recall that a pseudo-valuation on $K[x]$ is a map $K[x]\to \Q\cup\{\infty\}$ having the same properties as a valuation, except for the fact that the pre-image of $\infty$ is a prime ideal which is not necessarily zero.

\begin{lemma}\label{inequality}
Let $\mu_\infty$ be a pseudovaluation on $K[x]$, and let $\mu$ be an inductive valuation such that $\mu<\mu_\infty$. Let $\Phi_{\mu,\mu_\infty}$ be the set of all monic polynomials $\phi\in K[x]$ of minimal degree satisfying $\mu(\phi)<\mu_\infty(\phi)$. Then, there is a unique $\ll\in\mx(\Delta(\mu))$ such that  
$$
\Phi_{\mu,\mu_\infty}=\left\{\phi \in\kpm\mid \rr_{\mu}(\phi)=\ll\right\}.
$$
\end{lemma}

\begin{proof}
By \cite[Thm. 1.15]{Vaq}, any $\phi\in\Phi_{\mu,\mu_\infty}$ is a key polynomial for $\mu$ such that
$$
\phi\mmu g \sii \mu(g)<\mu_\infty(g),
$$
for any nonzero $g\in K[x]$. For any fixed $\phi\in\Phi_{\mu,\mu_\infty}$, Lemma \ref{mid=sim} shows that
$$
\Phi_{\mu,\mu_\infty}=\left\{\varphi \in \kpm\mid \varphi\smu\phi\right\}
$$
is the $\mu$-equivalence class of $\phi$ inside $\kpm$.
This ends the proof because, as seen in (\ref{repr}), the fibers of the map $\rr_\mu\colon \kpm\to \mx(\Delta(\mu))$
are the $\mu$-equivalence classes in $\kpm$.
\end{proof}

\begin{lemma}\label{values}
The group of values $\Gamma(\mu)$ of an inductive valuation $\mu$ satisfies
$$
\Gamma(\mu)=\{\mu(a)\mid a\in K[x],\ \deg a<e_\mu m_\mu\}.
$$ 
\end{lemma}

\begin{proof}
By (\ref{em}), $\mu$ admits a proper key polynomial $\phi$ of degree $e_\mu m_\mu$.
Consider a MacLane chain of $\mu$ as in (\ref{depth}) such that $\phi\nmid_\mu\phi_r$. Let $\mu'=[\mu;\phi,\nu]$ be any augmentation of $\mu$ determined by the choice of an arbitrary positive rational number $\nu$. The MacLane chain may be extended to a MacLane chain of length $r+1$ of $\mu'$ with last step $\mu\ \stackrel{\phi,\nu}\lra\ \mu'$.

Now, the claimed identity on $\Gamma(\mu)$ is proved in \cite[Lem. 3.2]{Rid}.
\end{proof}

\section{An example}\label{secExample}

Let $p$ be an odd prime number. Denote by $v$ the $p$-adic valuation on $\Q_p$ and let $\F=\Z/p\Z$. Consider the polynomial:
$$
f=x^4-2(p+p^2-p^3)\,x^2+p^2+2p^3-p^4-2p^5+p^6+p^8\in\Z[x].
$$

Let us apply the OM factorization method to compute the prime factors of $f$ in $\Z_p[x]$.

Clearly, $R_0(f)=\overline{f(y)}=y^4$. Thus, the type of order zero, $\ty_0=(y)$, divides all prime factors of $f$, and we have $\ord_{\ty_0}(f)=4$. We choose $\phi_1=x$ as a representative of $\ty_0$.

Let $\mu_0$ be the Gauss valuation extending $v$ to $\Q_p[x]$, introduced in section \ref{subsecGradedAlg}. The Newton polygon $N_{\mu_0,x}(f)$ is one-sided of length $4$ and slope $-1/2$. 

For the computation of residual polynomials we use the explicit recurrent method described in \cite[Sec. 3.1]{gen}. we have:
$$
R_{\mu_0,x,1/2}(f)=y^2-2y+1=(y-1)^2.
$$
Thus, we get a unique type of order one dividing all prime factors of $f$:
$$\ty_1=(y;(x,1/2,y-1)),$$but we now have $\ord_{\ty_1}(f)=2$. Hence, either $f$ is irreducible over $\Z_p[x]$, or it is the product $f=FG$ of two quadratic polynomials with $\ord_{\ty_1}(F)=\ord_{\ty_1}(G)=1$.

Take $\phi_2=x^2-p$ as a representative of $\ty_1$. The $\phi_2$-expansion of $f$ is:
\begin{equation}\label{exp2}
f=\phi_2^2-2(p^2-p^3)\,\phi_2+p^4-2p^5+p^6+p^8.
\end{equation}

The augmented valuation $\mu_1=[\mu_0;x,1/2]$ on $\Q_p[x]$ attached to $\ty_1$ acts on $\Q_p[x]$ as follows:
$$\mu_1\left(\sum_{0\le s}a_sx^s\right)=\mn\left\{\mu_0(a_s)+s/2\right\}=\mn\left\{v(a_s)+s/2\right\}.
$$
Since $\mu_1(\phi_2)=1$, the points $(s,\mu_1(a_s\phi_2^s))\in\R^2$ associated with the $ \phi_2$-expansion (\ref{exp2}) are $(2,2)$, $(1,3)$, $(0,4)$. Thus, $N_{\mu_1,\phi_2}(f)$ is one-sided of length $2$ and slope $-1$. The corresponding residual polynomial is:
$$
R_{\mu_1,\phi_2,1}(f)=y^2-2y+1=(y-1)^2.
$$
Again, we get only one type of order two dividing all prime factors of $f$:
$$\ty_2=(y;(x,1/2,y-1);(\phi_2,1,y-1)),$$still satisfying $\ord_{\ty_2}(f)=2$. Let us take $\phi_3=\phi_2-p^2=x^2-p-p^2$ as the simplest representative of $\ty_2$. The $\phi_3$-expansion of $f$ is:
\begin{equation}\label{exp3}
f=\phi_3^2+2p^3\,\phi_3+p^6+p^8.
\end{equation}

The non-normalized valuation $\mu_2=[\mu_1;\phi_2,1]$ attached to $\ty_2$ acts on $\Q_p[x]$ as follows:
$$\mu_2\left(\sum_{0\le s}a_s\phi_2^s\right)=\mn\left\{\mu_1(a_s)+2s\right\}.
$$
Since $\mu_2(\phi_3)=2$, the points in $\R^2$ associated with the $ \phi_3$-expansion (\ref{exp3}) are $(2,4)$, $(1,5)$, $(0,6)$. Thus, $N_{\mu_2,\phi_3}(f)$ is again one-sided of length $2$ and slope $-1$. The corresponding residual polynomial is:
$$
R_{\mu_2,\phi_3,1}(f)=y^2+2y+1=(y+1)^2.
$$
Again, we get only one type of order three dividing all prime factors of $f$:
$$\ty_3=(y;(x,1/2,y-1);(\phi_2,1,y-1);(\phi_3,1,y+1)),$$still satisfying $\ord_{\ty_3}(f)=2$. Let us take $\phi_4=\phi_3+p^3=x^2-p-p^2+p^3$ as a representative of $\ty_3$. The $\phi_4$-expansion of $f$ is:
\begin{equation}\label{exp4}
f=\phi_4^2+p^8.
\end{equation}

The valuation $\mu_3=[\mu_2;\phi_3,1]$ attached to $\ty_3$ acts on $\Q_p[x]$ as follows:
$$\mu_3\left(\sum_{0\le s}a_s\phi_3^s\right)=\mn\left\{\mu_2(a_s)+3s\right\}.
$$
Since $\mu_3(\phi_4)=3$, the points in $\R^2$ associated with the $ \phi_4$-expansion (\ref{exp4}) are $(2,6),\, (0,8)$. Thus, $N_{\mu_3,\phi_4}(f)$ is again one-sided of length $2$ and slope $-1$. The corresponding residual polynomial is:
$$
R_{\mu_3,\phi_4,1}(f)=y^2+1.
$$
The factorization of this polynomial in $\F[y]$ depends on the class of $p$ modulo $4$. The method proceeds in a different way according to this class.\newpage

\noindent{\bf Case $p\equiv -1\md{4}$. }\medskip

The polynomial $y^2+1$ is irreducible in $\F[y]$ and we get a unique type of order four dividing all prime factors of $f$:
$$\ty_4=(y;(x,1/2,y-1);(\phi_2,1,y-1);(\phi_3,1,y+1);(\phi_4,1,y^2+1)),$$for which $\ord_{\ty_4}(f)=1$. This implies that $f$ is irreducible in $\Z_p[x]$. Also, if $L/\Q_p$ is the finite extension of $\Q_p$ determined by $f$, we have
$$
e(L/\Q_p)=e_1e_2e_3e_4=2,\quad f(L/\Q_p)=f_0f_1f_2f_3f_4=2,
$$
where $e_i$ are the lowest term denominators of the slopes of $\ty_4$ and $f_i$ are the degrees of the $\psi$-polynomials of all levels of $\ty_4$.

However, the information about $f$ we have been collecting in the type $\ty_4$ is not intrinsic. It depends on the choices of representatives for the types
$\ty_0$, $\ty_1$, $\ty_2$, $\ty_3$.
Let us consider the following optimal type equivalent to $\ty_4$:
$$
\ty=(y;(x,1/2,y-1);(\phi_4,3,y^2+1)),
$$
obtained by an iterative application of Lemma \ref{optstep}.

By Theorem \ref{finalchar}, $\ord_\ty(f)=1$ and $f$ is a representative of $\ty$. Moreover, since the type $\ty$ is strongly optimal, the equivalence class of $\ty$ is the canonical class attached to the Okutsu class of $f$ through the mapping of (\ref{main}). 

Therefore, the data supported by $\ty$ are intrinsic data of $f$. For instance, the \emph{Okutsu depth} of $f$ is two and $[x,\,\phi_4]$ is an \emph{Okutsu frame} of $f$ \cite{okutsu}. This means that
$$
\begin{array}{rl}
\frac 12 = v(\t)\ge v(h(\t)),&\mbox{ for all monic }h\in\Z_p[x]\mbox{ with }\deg h<2,\\
3 = v(\phi_4(\t))\ge v(h(\t)),&\mbox{ for all monic }h\in\Z_p[x]\mbox{ with }\deg h<4,
\end{array}
$$
where $\t$ is a root of $f$ in $\overline{\Q}_p$. In particular, the slopes $1/2$ and $3$ are intrinsic data of $f$.\medskip

\noindent{\bf Case $p\equiv 1\md{4}$. }\medskip

The polynomial $y^2+1$ splits as $(y-i)(y+i)$ in $\F[y]$, where $i\in\F$ satisfies $i^2=-1$. We get then two inequivalent types dividing $f$:
$$
\begin{array}{c}
\ty_4=(y;(x,1/2,y-1);(\phi_2,1,y-1);(\phi_3,1,y+1);(\phi_4,1,y-i)),\\
\ty'_4=(y;(x,1/2,y-1);(\phi_2,1,y-1);(\phi_3,1,y+1);(\phi_4,1,y+i)),
\end{array}
$$
with $\ord_{\ty_4}(f)=\ord_{\ty'_4}(f)=1$. This implies that $f=FF'$ splits in $\Z_p[x]$ into the product of two monic quadratic irreducible polynomials $F$, $F'$ such that
$$
\ord_{\ty_4}(F)=1,\ \ord_{\ty'_4}(F)=0;\quad \ord_{\ty_4}(F')=0,\ \ord_{\ty'_4}(F')=1. 
$$
If $L/\Q_p$, $L'/\Q_p$ are the quadratic extensions of $\Q_p$ determined by these prime factors, we have
$$
e(L/\Q_p)=e(L'/\Q_p)=2,\quad f(L/\Q_p)=f(L'/\Q_p)=1.
$$

Also, by taking representatives of these types we obtain concrete Okutsu approximations to the unknown factors $F$, $F'$:
\begin{equation}\label{gg}
\as{1.2}
\begin{array}{l}
G:=\phi_4-ip^4=x^2-p-p^2+p^3-ip^4\approx F,\\
G':=\phi_4+ip^4=x^2-p-p^2+p^3+ip^4\approx F',
\end{array}
\end{equation}
where now $i\in\Z$ is an arbitrary lifting of $i\in\F$.

Again, the information about $F$, $F'$ contained in the types $\ty_4$, $\ty'_4$, respectively, is not intrinsic. Consider the optimal types equivalent to $\ty_4$, $\ty'_4$, respectively:
$$
\ty=(y;(x,1/2,y-1);(\phi_4,3,y-i)),\qquad 
\ty'=(y;(x,1/2,y-1);(\phi_4,3,y+i)),
$$
obtained by an iterative application of Lemma \ref{optstep}.

By Theorem \ref{finalchar}, these types satisfiy
$$
\ord_{\ty}(F)=1,\ \ord_{\ty'}(F)=0;\quad \ord_{\ty}(F')=0,\ \ord_{\ty'}(F')=1,
$$
and the polynomials $G$, $G'$ of (\ref{gg}) are representatives of $\ty$, $\ty'$, respectively.

Caution! The types $\ty$ and $\ty'$ are optimal, but not strongly optimal. Hence, the information contained in the last level of $\ty$, $\ty'$ is not intrinsic either. In this case, the equivalence class of strongly optimal types associated with the Okutsu class of $F$ is the class of the type $\ty_1$. In fact, by Lemma \ref{optstep}, $F$ and $F'$ are representatives of $\ty_1$. This means that the prime polynomials $F$, $F'$ both correspond to the same strongly optimal type $\ty_1$ by the mapping of (\ref{main}); hence, these polynomials are Okutsu equivalent. Actually, if we denote by $[g]$ the Okutsu class of a prime polynomial $g\in\oo_v[x]$, we have:
$$
[x^2-p]=[\phi_4]=[F]=[F']=[G]=[G'],
$$
and all these polynomials determine the same quadratic extension of $\Q_p$. In general, the extensions determined by two Okutsu equivalent prime polynomials in $\Z_p[x]$ have isomorphic maximal tamely ramified subextensions \cite{okutsu}.  

The type $\ty_1$ contains intrinsic information about all these Okutsu equivalent prime polynomials in $\Z_p[x]$. They all have Okutsu depth one, the family $[x]$ is an Okutsu frame and the slope $1/2$ has the following intrinsic meaning:
$$
\frac 12 = v(\t)\ge v(h(\t)),\mbox{ for all monic }h\in\Z_p[x]\mbox{ with }\deg h<2.
$$

This situation enlightens an important feature of the OM factorization algorithm. When some prime factors of the input polynomial are in the same Okutsu class, the algorithm computes first the common strongly optimal (equivalence class of the) type attached to them, but then it must work further to enlarge this type with an adequate last level which enables one to distinguish the different prime factors.

\end{document}